\documentclass[english]{article} 
\usepackage[utf8]{inputenc} 
\usepackage[T1]{fontenc} 
\usepackage[a4paper]{geometry} 
\usepackage{amsmath, amssymb, mathtools, amsthm}
\usepackage{mathtools} 
\usepackage{mathrsfs}  
\usepackage{graphicx} 
\usepackage{xcolor}   
\usepackage[english]{babel} 
\usepackage{bbm}
\usepackage[title]{appendix}
\newcommand{\eps}{\varepsilon}
\usepackage{bbold}
\usepackage{ stmaryrd }
\usepackage{dsfont} 
\usepackage{caption}
\usepackage{subcaption}
\usepackage{comment}
\usepackage{enumitem} 
\usepackage{geometry}
\usepackage{tikz}
\usepackage{tkz-graph}
\usepackage{float}
\usepackage{hyperref} 

\newcommand{\red}{\textcolor{red}}
\newcommand \commentout[1] {}

\newcommand{\R}{\mathbb{R}}
\newcommand{\N}{\mathbb{N}}

\newcommand {\al} {\alpha}

\newcommand {\sg} {\sigma}
\newcommand {\vp} {\varphi}
\newcommand {\lb} {\lambda}

\newcommand {\Chi} {{\bf \raise 2pt \hbox{$\chi$}} }

\newcommand {\Div}  { {\rm div} }
\newcommand {\f}   {\frac}
\newcommand {\p}   {\partial}

\newcommand*{\dd}{\mathop{\kern0pt\mathrm{d}}\!{}}
\newcommand {\intM}   {\int_{0}^{M}}
\newcommand{\diff}{\mathop{}\!\mathrm{d}}

\DeclarePairedDelimiter{\norm}{\lVert}{\rVert}

\newtheorem{mainthm}{Theorem}

\theoremstyle{plain}
\newtheorem*{thm*}{Theorem}
\newtheorem{thm}{Theorem}[section]

\newtheorem{lemma}[thm]{Lemma}
\newtheorem{proposition}[thm]{Proposition}

\theoremstyle{remark}
\newtheorem{remark}[thm]{\bf Remark}
\newtheorem{definition}[thm]{\bf Definition}

\newcommand{\ie}{\emph{i.e.}\;}

\newcommand{\eg}{\emph{e.g.}\;}
\newcommand{\etal}{\emph{et al.}\;}

\newcommand{\beq}{\begin{equation}}
\newcommand{\eeq}{\end{equation}}
\newcommand{\bea} {\begin{array}{rl}}
\newcommand{\eea} {\end{array}}
\newcommand{\bepa}{\left\{ \begin{array}{l}}
\newcommand{\eepa} {\end{array}\right.}

\numberwithin{equation}{section}

\title{Pressure jump and radial stationary solutions of the degenerate Cahn-Hilliard equation}

\author{Charles Elbar\thanks{Sorbonne Universit\'{e}, CNRS, Universit\'{e} de Paris, Inria, Laboratoire Jacques-Louis Lions, F-75005 Paris, France } \thanks{email: charles.elbar@sorbonne-universite.fr}
\and Beno\^it Perthame\footnotemark[1]  \thanks{email: benoit.perthame@sorbonne-universite.fr}
\and Jakub Skrzeczkowski\thanks{Faculty of Mathematics, Informatics and Mechanics, University of Warsaw, Stefana
Banacha 2, 02-097 Warsaw, Poland} \thanks{email: jakub.skrzeczkowski@student.uw.edu.pl}
}

\begin{document}

\maketitle

\begin{abstract}
 The Cahn-Hilliard equation with degenerate mobility is used in several areas including the modeling of living tissues. We are interested in quantifying the pressure jump at the interface in the case of incompressible flows. To do so, we include an external force and consider stationary radial solutions. This allows us to compute the pressure jump in the small dispersion regime. We also characterize compactly supported stationary solutions in the  incompressible case, prove the incompressible limit and prove convergence of the parabolic problems to stationary states. 

\end{abstract}
\vskip .7cm

\noindent{\makebox[1in]\hrulefill}\newline
2010 \textit{Mathematics Subject Classification.}  35B40; 35B45; 35G20 ; 35Q92
\newline\textit{Keywords and phrases.} Degenerate Cahn-Hilliard equation;  Asymptotic Analysis; Incompressible limit; Hele-Shaw equations; Surface tension; Pressure jump.

\section{Introduction}

The degenerate Cahn-Hillard equation is now commonly used in tumor growth modeling and takes into account surface tensions at the interface between different types of cells, leading to a jump of pressure. In order to compute this jump, we propose to set the problem in a spherically symmetric domain with a boundary determined by the radius $R_b$, and to include an external force. Therefore we consider, in two dimensions for simplicity, the equation
%
\begin{align}
\f{\p (rn)}{\p t}-\f{\p}{\p r}\left(rn\f{\p(\mu+V)}{\p r}\right)=0,\quad \text{in}\quad &(0,+\infty)\times I_{R_{b}},\label{eq:CHR1}\\
\mu =n^{\gamma}-\f{\delta}{r}\f{\p}{\p r}\left(r\f{\p n}{\p r}\right),\quad \text{in}\quad &(0,+\infty)\times I_{R_{b}} \label{eq:CHR2},    
\end{align}
where $I_{R_{b}} = (0,R_{b})$ is the line segment of length $R_{b}$. Equations \eqref{eq:CHR1}--\eqref{eq:CHR2} are equipped with Neumann boundary conditions
\begin{equation}  \label{eq:Neumann}
\f{\p n}{\p r}\Big|_{r=0}=\f{\p n}{\p r}\Big|_{r=R_{b}}=n\f{\p(\mu+V)}{\p r}\Big|_{r=0}=n\f{\p(\mu+V)}{\p r}\Big|_{r=R_{b}}=0,
\end{equation}
and with an initial condition satisfying
\begin{equation}  \label{eq:IDR}
n_0\in H^{1}(I_{R_{b}}), \qquad n_{0}\ge 0.
\end{equation}

We only consider nonnegative solutions and thus the term $n^\gamma$ is well defined and the (normalized by a factor $\frac \pi 2$) total mass is
\begin{equation}\label{def:mass}
m := \int_0^{R_b} r\, n_{0}(r) \diff r =\int_0^{R_b} r\, n(t,r) \diff r.
\end{equation}
Finally, the confining potential $V(r)$ is of class $C^{1}$.
\\
 Our first result concerns the existence of solutions of \eqref{eq:CHR1}-\eqref{eq:CHR2}, their regularity, and asymptotic behaviour.
\begin{mainthm}[Existence of solutions and long term asymptotic]
\label{thm:finalthm1}
There exists a global weak solution of \eqref{eq:CHR1}-\eqref{eq:IDR} in the sense of Definition \ref{def:weak} and it satisfies estimates as in Remark~\ref{rem:more_prop_weak_sol}. Moreover, up to a subsequence, $\{r\, n(t+k,r)\}_{k}$ converges
locally in time uniformly in space to a stationary solution $ r \, n_{\infty}(r)\geq 0$ where $n_{\infty} \in C^1(\overline{I_{R_{b}}})$ satisfies $m= \int_0^{R_b} r\, n_{\infty}(r) \diff r$ and
\begin{equation}\label{eq:thm1_stationary}
rn_{\infty}\f{\p(\mu_{\infty}+V)}{\p r} =0, \qquad \quad \mu_{\infty} =n^{\gamma}_{\infty}-\f{\delta}{r}\f{\p}{\p r}\left(r\f{\p n_{\infty}}{\p r}\right)\qquad \quad n'_{\infty}(0) = n'_{\infty}(R_{b}) = 0.
\end{equation}
\end{mainthm}

Our second result characterizes possible stationary states and shows we can distinguish an interval where $n_\infty=0$ and another where $\mu_{\infty}+V$ is constant as expected from the first equation in~\eqref{eq:thm1_stationary}. From now on, we consider the confining potential $V(r)=r^{2}$ for simplicity. The proof may be adapted to any increasing potential. 

\begin{mainthm}[Characterization of the stationary states]
\label{thm:finalthm2}
 Let $n_{\infty} \in C^1([0,R_{b}])$, $n_{\infty} \geq 0$, be a solution of \eqref{eq:thm1_stationary} as built in Theorem~\ref{eq:thm1_stationary}.
 \\
 \textbf{(A)} Then, $n_{\infty}$ is nonincreasing and it satisfies $0 \leq n_{\infty}(R_{b}) < \frac{2m}{R_{b}^{2}}$.
 \\
 \textbf{(B)} Assume $n_{\infty}(R_{b}) = 0$ and let $R > 0$ be the smallest argument such that $n_{\infty}(R)=0$ and thus $n_{\infty} > 0$ in $[0,R)$. Then, there is $\lambda_{\infty} \in (0,R^2)$ such that 
\begin{equation}\label{eq:model}
\begin{cases}
n_{\infty}^{\gamma} - \f{\delta}{r}n_{\infty}'-\delta n_{\infty}'' = R^2 - r^2 - \lambda_{\infty} &\mbox{ in } (0,R),
\\
n_{\infty}(R)=n_{\infty}'(R)=0.
\end{cases}
\end{equation}
and, given $R>0$, there is at most one couple $(n, \lambda)$ solving \eqref{eq:model}.
\\
\textbf{(C)} Fix $\delta \in (0,1)$. There exists $R(m)$, independent of $\gamma$, such that when $R_b > R(m)$, then $n_{\infty}(R_b) = 0$.
\end{mainthm}

%

Next, we focus on the incompressible limit of the solutions of~\eqref{eq:model}, that is when  $\gamma_k \to \infty$. We denote by $n_k$ the steady state associated with $\gamma_k$ and assume that $R_b$ is large enough so that $n_k(R_b)=0$.

\begin{mainthm}[Incompressible limit of the stationary states]
\label{thm:finalthm3}
Let $\{\gamma_k\}_{k \in \mathbb{N}}$ be any sequence such that $\gamma_k \to \infty$.  Let $\{n_k\}_{k \in \mathbb{N}}$ be a sequence of stationary states with the same mass $m$ and with radius $R_{k}$, being the smallest argument such that $n_{k}(r)=0$. 

Then, $n_k \to n_{inc}$  strongly and $R_{k}\to R$, where $n_{inc}$ and $R$ are uniquely defined in Proposition~\ref{mainthm3}.
Moreover,  the sequence of pressures $\{p_k:=n_{k}^{\gamma_{k}}\}_{k\in\N}$ converges weakly to some pressure $p_{inc}$ such that $p_{inc}(n_{inc}-1)=0$ and $p_{inc}$ has a jump at
$\p\{n_{inc}=1\}$ 
$$
\llbracket p_{inc}\rrbracket \approx \sqrt[3]{6} \, R^{2/3} \, \delta^{1/3},  \qquad \mbox{as }\delta \to 0.
$$ 
\end{mainthm}

\noindent The profile $n_{inc}$ obtained for the incompressible limit of stationary states is depicted in Figure~\ref{plot:K}. The density is equal to $1$ on a certain interval $(0, R_0)$ where the pressure is positive. Then, the pressure vanishes and the density  decreases to $0$ on a small interval $(R_0, R)$. At the boundary point $R_0$ the pressure undergoes a jump, which depends on the surface tension coefficient $\delta$ and on the shape of the confinement potential $V$. More precisely, for a general potential $V(r)$, this jump is determined by

\begin{equation}\label{eq:jump_pressure}
 \llbracket p_{inc} \rrbracket\approx\frac{\sqrt[3]{12}}{2} \,\delta^{1/3} \, (V'(R))^{2/3}\quad \text{as $\delta\to 0$},
\end{equation}
where $R$ is the smallest value where $n(R)=0$ and we have the estimate $R^{2}-R_{0}^{2} \approx \f{2\sqrt[3]{12}\delta^{1/3}R}{\sqrt[3]{V'(R)}}$. 
We point out that the limiting profile (including parameters $R_0$ and $R$) is uniquely determined in terms of mass $m$, $\delta$ and $V$  cf. Proposition \ref{mainthm3}.\\

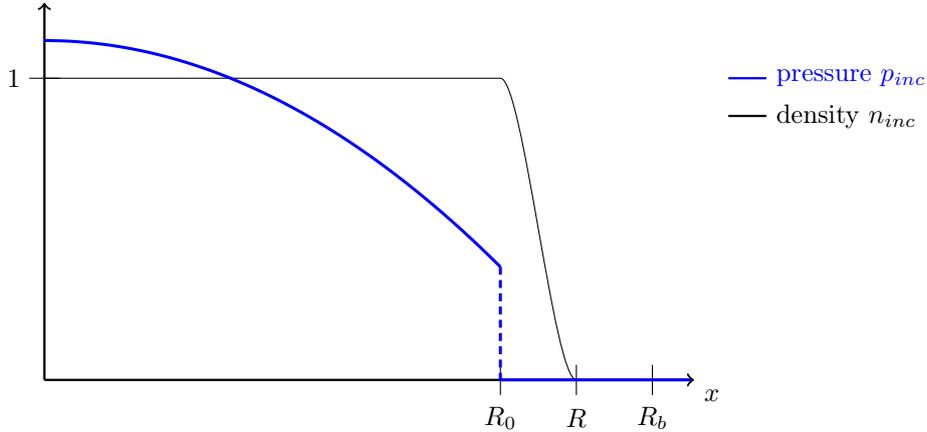
\begin{figure}
\begin{center}

\begin{tikzpicture}

\draw[line width=0.3mm,->] (0,0) -- (8.55,0) node[anchor=north west] {$x$};
\draw[line width=0.3mm,->] (0,0) -- (0,5);

\node at (-0.4, 4.0) {$1$};
\draw (-0.2,4.0) -- (0.2,4.0);

\node at (6,-0.5) {$R_0$};
\draw (6,-0.2) -- (6,0.2);
\node at (7,-0.5) {$R$};
\draw (7,-0.2) -- (7,0.2);
\node at (8,-0.5) {$R_b$};
\draw (8,-0.2) -- (8,0.2);

\draw (0,4) -- (6,4);
\draw (6,4) .. controls (6.3,4) and (6.7, 0) .. (7,0);

\draw[scale=1, line width=0.4mm, domain=0:6, smooth, variable=\x, blue] plot ({\x}, {4.5-\x*\x/12});
\draw[scale=1, line width=0.4mm, domain=6:8.50, smooth, variable=\x, blue] plot ({\x}, {0});
\draw[densely dashed,line width=0.4mm, blue] (6,1.5) -- (6,0);

\draw[line width=0.3mm,-, blue] (9,4)--(9.5,4) node[anchor=west] {pressure $p_{inc}$};
\draw[line width=0.3mm,-] (9,3.5)--(9.5,3.5) node[anchor=west] {density $n_{inc}$};

\end{tikzpicture}
\vspace{-1mm}
\caption{Plot of the limiting profile $n_{inc}$, as $\gamma_k \to \infty$, for the potential $V(r)=r^{2}$. We can observe that the pressure has a discontinuity at $R_0$ with  $(0,R_0)=\{n_{inc}=1\}=\{p_{inc}>0\}$, while the density remains $C^1$.}
\label{plot:K}
\end{center}
\end{figure}

In the above statements, the main novelty concerns the incompressible limit $\gamma \to \infty$ for the stationary states. A previous work in this direction \cite{elbar-perthame-poulain} made use of viscosity relaxation, which provided additional estimates implying compactness. In our case, assuming the radial symmetry of the problem, we are able to characterize the incompressible limit of the sequence of compactly supported stationary solutions. While our setting is restrictive, it allows performing many computations explicitly. In particular, we find how the pressure jump depends on $V$ and $\delta$, cf. \eqref{eq:jump_pressure}.\\

\noindent \textbf{Open question.} In this paper, we prove that the stationary states are compactly supported or at least zero on the boundary if the domain is large enough. It is logical to ask whether the solutions of the parabolic equation are compactly supported for a large domain and a strong confining potential. This question is still open. However, a work in this direction \cite{finite_speed} has proved that in dimension $1$, one could expect the solutions of the Cahn-Hilliard equation without confining potential to propagate with finite speed. By adding this potential, we can expect to have a better result, and compactly supported solutions, with time-independent support.

\paragraph{Contents of the paper.}
The above theorems are proved in the following sections. Section~\ref{sect:existence}, is devoted to prove Theorem~\ref{thm:finalthm1}. In Section~\ref{sect:stationary_states} we prove Theorem \ref{thm:finalthm2} and in Section~\ref{sec:proof3} we give the proof of our main new result, namely Theorem~\ref{thm:finalthm3}. Numerical simulations of the model with a source term and no confining potential are presented in Section~\ref{sec:num}. The appendix contains the computation of the pressure jump for a general confining potential. \\

\paragraph{Notations.}

For a function $n(x,t)$ we associate a function in radial coordinates that is still denoted by $n(r,t)$. 
For $1\le p,s\le +\infty$ or $s=-1$ and $\Omega$ a domain, $L^{p}(\Omega), H^{s}(\Omega)$ denote the usual Lebesgue and Sobolev spaces. When $s=-1$, $H^{-1}(\Omega)$ is the topological dual of $H^{1}(\Omega)$. Here $H^{s}(\Omega)=W^{s,2}(\Omega)$ in the usual notation. We also consider the Bochner spaces $L^{p}(0,T;H^{s}(\Omega))$ associated with the norm
\begin{equation*}
    \norm{f}_{L^{p}(0,T;H^{s}(\Omega))}=\left(\int_{0}^{T}\norm{f}_{H^{s}(\Omega)}^{p}\right)^{1/p} .
\end{equation*}
The partial derivative with respect to the radial variable is written as  $\p_{r}u(r)=\f{\p u}{\p r}(r)=u'(r)$. Finally, $C$ denotes a generic constant which appears in inequalities and whose value can change from one line to another. This constant can depend on various parameters unless specified otherwise.

\subsection{Literature review and biological relevancy of the system}

\paragraph{Tissue growth models and Hele-Shaw limits.} Development of tissue growth models is presently a major line of research in mathematical biology.
Nowadays, number of models are available \cite{byrne_modelling_2004,Friedman,Lowengrub-bridging} with the common feature that they use the tissue internal pressure as the main driver of both the cell movement and proliferation. The simplest example of a mechanical model of living tissue is the {\em compressible} equation
\begin{equation}
    \p_t n = \Div\left(n\nabla p \right) + nG(p),\quad p=P_\gamma(n):=n^{\gamma}, 
    \label{eq:general-mod}
\end{equation}
in which $p(t,x)=P(n(t,x))$, with $P$ a law of state, is the pressure and $n$ the density of cell number. 
Here, the cell velocity is given via Darcy's law which captures the effect of cells moving away from regions of high compression. Dependence on growth function pressure has also been used to model the sensitivity of tissue proliferation to compression (contact inhibition, \cite{BD}).

An important problem is to understand the so-called incompressible limit  (\ie $\gamma\to \infty$) of this model. Perthame \etal~\cite{Perthame-Hele-Shaw} have shown that in this limit, solutions of \eqref{eq:general-mod} converge to a limit solution $(n_\infty,p_\infty)$ of a Hele-Shaw-type free boundary limit problem for which the speed of the free boundary is given by the normal component of $\nabla p_\infty$, see also other approaches in~\cite{KT18, LiuXu}. In this limit, the solution of \eqref{eq:general-mod} is organized into 2 regions: $\Omega(t)$ in which the pressure is positive (corresponding to the tissue) and outside of this zone where $p=0$. Furthermore, the free boundary problem is supplemented by a complementary equation that indicates that the pressure satisfies
\begin{equation}
    -\Delta p_\infty = G(p_\infty),\quad \text{in}\quad \Omega(t),\quad \text{or similarly}\quad p_\infty(\Delta p_\infty +  G(p_\infty)) = 0 \quad \text{a.e. in } \Omega.
    \label{eq:complem}
\end{equation}

In this model, the pressure stays continuous in space, with jumps in time, and is equal to 0 at the interface. This is because only repulsive forces were taken into account. Hence, the crucial role of the cell-cell adhesion and thus the pressure jump at the surface of the tissue is not retrieved at the limit. Additionally, as pointed out by Lowengrub~\etal~\cite{Lowengrub-bridging}, the velocity of the free surface should depend on its geometry and more precisely on the local curvature denoted by~$\kappa$. 

This motivated considering variants of the general model~\eqref{eq:general-mod}, where other physical effects of mechanical models of tissue growth are introduced. 
One of them is the addition of the effect of viscosity in the model, which has been made to represent the friction between cells~\cite{Basan,Bittig_2008} through the use of Stokes' or Brinkman's law. 
Moreover, as pointed out by Perthame and Vauchelet~\cite{Perthame-incompressible-visco}, Brinkman's law leads to a simpler version of the model and, therefore, is a preferential choice for its mathematical analysis. 
Adding viscosity through the use of Brinkman's law leads to the model  
\begin{equation}
\begin{cases}
    \p_t n = \Div\left(n\nabla \mu \right) + nG(p), \quad &\text{in}\quad (0,+\infty)\times\Omega,\\
    -\sg \Delta \mu + \mu = p,\quad &\text{in}\quad (0,+\infty)\times\Omega.
    \end{cases}
    \label{eq:visco-living}
\end{equation}
The incompressible limit of this system also yields the complementary relation (see~\cite{Perthame-incompressible-visco})
\[
    p_\infty(p_\infty - \mu_\infty - \sigma G(p_\infty)) = 0,\quad \text{a.e. in } \Omega.
\]
In the incompressible limit, notable changes compared to the system with Darcy's law are found. First, the previous complementary relation is different compared to Equation~\eqref{eq:complem}, and the pressure $p_\infty$ in the limit is discontinuous, \ie there is a jump of the pressure located at the surface of $\Omega(t)$.
However, the pressure jump is related to the potential $\mu$ and not to the local curvature of the free boundary $\p \Omega(t)$. The authors already indicated that a possible explanation for this is that the previous model does not include the effect of surface tension.

\paragraph{Surface tension and pressure jump.}

Surface tension is a concept associated with the internal cohesive forces between the molecules of a fluid: hydrogen bonds, van der Waals forces, metallic bonds, etc. Inside the fluid, molecules are attracted equally in all directions leading to a net force of zero; however molecules on the surface experience an attractive force that tends to pull them to the interior of the fluid: this is the origin of the surface energy. This energy is equivalent to the work or energy required to remove the surface layer of molecules in a unit area. The value of the surface tension will vary greatly depending on the nature of the forces exerted between the atoms or molecules. In the case of solid tumor cells in a tissue, it reflects the cell-cell adhesion tendency between the cells and depends on the parameter $\delta$ and the geometry of the tumor. 

In the previous definition, the surface tension is associated with a single body that has an interface with the vacuum. When one considers two bodies, the surface energy of each body is modified by the presence of the other and we speak of interfacial tension. The latter depends on the surface tension of each of the two compounds, as well as the interaction energy between the two compounds. In the system considered above, it is then possible to imagine that the vacuum in which the tumor grows is in fact another body that has an internal pressure of the form $V(r)$ which increases with respect to $r$ so that the tumor is stopped at some point and we can consider the stationary states. 

For such a tumor to be in equilibrium, it is necessary that the interior is overpressured relative to the exterior by an amount. This amount is called the pressure jump and is computed explicitly in our case. 

Surface tension effects can be introduced in  the Hele-Shaw model as follows (see \eg \cite{Escher-solutions})
\begin{equation}
    \begin{cases}
        -\Delta \mu = 0\quad \text{in}\quad \Omega \setminus \p\Omega(t),\\
        \mu = \sigma \kappa \quad \text{on} \quad \p\Omega(t).
    \end{cases}
\end{equation}
where $\sigma$ is a positive constant, called a surface tension and $\kappa$ is a mean curvature of $\partial \Omega(t)$. This correct Hele-Shaw limit has been formally obtained as the sharp-interface asymptotic model of the Cahn-Hilliard equation~\cite{Alikakos_convergence}; see also~\cite{Chen1996} for a convergence result in a weak varifold formulation. This suggests that the Cahn-Hillard equation is an appropriate model to capture surface tension effects.

\paragraph{The Cahn Hilliard equation.}

Cahn-Hillard type models for tissue growth have been developed based on the theory of mixtures in mechanics, see \cite{Frieboes-2010-CH,chatelain_2011,sciume}  and the references therein. Nowadays, they are widely used, in particular for tumor growth, and analysed,  \cite{Garcke-2016-CH-darcy,Agosti-CH-2017,Garcke-2018-multi-CH-darcy,Frigeri-CH-Darcy,ebenbeck_cahn-hilliard-brinkman_2018,Ebenbeck-Brinkman,perthame_poulain}. Originally introduced in the context of materials sciences \cite{Cahn-Hilliard-1958}, they are currently applied in numerous fields, including complex fluids, polymer science, and mathematical biology.  For the overview of mathematical theory, we refer to~\cite{MR4001523}.

Usually, in mechanical models, the Cahn-Hillard equation takes the form  
\begin{equation}
    \p_t \vp = \Div\big(b(\vp)\nabla \left(\psi^\prime(\vp)-\delta \Delta \vp \right)\big) \Longleftrightarrow \begin{cases}
        \p_t \vp &= \Div\left(b(\vp)\nabla \mu \right),\\
        \mu &= -\delta \Delta \vp +\psi^\prime(\vp),
    \end{cases}
    \label{eq:Cahn-Hilliard}
\end{equation}
where $\vp$ represents the relative density of cells $\vp=n_1/(n_1+n_2)$, $b$ is the mobility, $\psi$ is the potential while $\mu$ is the quantity of chemical potential, which is a quantity related to the effective pressure. From the point of view of mathematical biology, the most relevant case is $b(\vp) = \vp(1-\vp)$, which is referred to as degenerate mobility.
\\

In our context, \eqref{eq:CHR1} models the motion of a population of cells constituting a biological tissue in the form of a continuity equation. It takes into account pressure, the surface tension occurring at the surface of the tissue and its viscosity. More precisely,
the equation for $\mu$ (\ie equation~\eqref{eq:CHR2}) includes the effects of both the pressure, through the term $n^\gamma$ with $\gamma > 1$ that controls the stiffness of the pressure law, and surface tension by $- \delta \Delta n$, where $\sqrt\delta$ is the width of the interface in which partial mixing of the two components $n_1$, $n_2$ occurs. 

A similar Cahn-Hilliard problem, without radial symmetry assumption, has previously been considered in~\cite{elbar-perthame-poulain}, but including a relaxation (viscosity) term and a proliferation source term in place of the confinement potential. In the incompressible limit, the authors obtain a jump in pressure at the interface at all times for the relaxed system. The aim here is to justify a rigorous limit without viscosity relaxation and mostly to compute the pressure jump by analyzing the stationary states of a system with confining potential.

\section{Existence, regularity, and long term behavior}
\label{sect:existence}

The existence of weak solutions for the Cahn-Hilliard equation with degenerate mobility usually follows the method from~\cite{Garcke-CH-deg}. The idea is to apply a Galerkin scheme with a non-degenerate regularized mobility, \ie, calling $b(n)$ the mobility, then one considers an approximation $b_{\eps}(n)\ge\eps$. Then, using standard compactness methods one can prove the existence of weak solutions for the initial system. However, the uniqueness of the weak solutions is still an open question.

In the case of a radially symmetric solution, the resulting system has only one dimension in space, and it is possible to apply a fixed point theorem, see~\cite{Yin2001}, to obtain better regularity results.  Since the solutions have radial symmetry, the equation is singular at $r=0$. Therefore, the first step is to consider the system with $r+\eps$ instead of $r$ and a regularized mobility. The existence of solutions for a similar regularized system has been achieved in~\cite{Yin2001} based on a result of~\cite{JPDE-7-77} and we do not repeat the arguments here. Then, we can pass to the limit $\eps\to 0$. Finally, the nonnegativity of the limiting solution is achieved with the bounds provided by the entropy.

We finally point out that since we are also interested in the convergence to the stationary states, one needs to carefully verify that the bounds do not depend on time.

\begin{definition}[Weak solutions]
\label{def:weak}
We say that $n(t,r)$ is a global weak solution of the equation~\eqref{eq:CHR1}-\eqref{eq:CHR2} provided that  
\begin{itemize}
    \item $n$ is nonnegative,
    \item $r n$ is continuous in $[0,\infty)\times \overline{I_{R_{b}}}$, $\sqrt{r}\, n \in L^{\infty}((0,\infty)\times I_{R_{b}})$ and $r\,\p_{t}n\in L^{2}((0,\infty);H^{-1}(I_{R_{b}}))$,
    \item $\sqrt{rn}\p_{r}\mu\in  L^{2}((0,\infty) \times \overline{I_{R_{b}}}\setminus\{rn=0\})$ and $\mu$ is defined in~\eqref{eq:CHR2},
    \item for every test function $\varphi\in L^{2}((0,\infty);H^{1}(I_{R_{b}}))\cap C^{1}_c([0,\infty)\times {I_{R_{b}}})$
\begin{equation*}
\int_{0}^{T}r\langle\p_{t}n,\varphi\rangle_{H^{-1},H^{1}} \diff t +\int_0^T \int_{0}^{R_{b}}\mathds{1}_{rn>0}\, r\, n\, \p_{r}(\mu+V)\p_{r}\varphi \diff r \diff t=0,
\end{equation*} 
and 
\begin{equation*}
    \int_{0}^{T}r\langle\p_{t}n,\varphi\rangle_{H^{-1},H^{1}} \diff t =-\int_0^T \int_{0}^{R_{b}} rn\p_{t}\varphi \diff r \diff t -\int_{0}^{R_{b}} \varphi(0,r) n_{0}(r) \diff r,
\end{equation*}
\item $n'(t, R_{b})=0$ for a.e. $t\in(0,T)$.
\end{itemize}

\end{definition}

\begin{remark}[Energy, entropy properties of weak solutions]\label{rem:more_prop_weak_sol}
In fact, we construct solutions satisfying additionally mass, energy, and entropy relations as follows: for a.e. $\tau \in [0,T]$
\begin{equation}
\label{eq:cons_mass_limit}
\int_{0}^{R_{b}} r\, n(\tau,r) \diff r = \int_{0}^{R_{b}} r\, n_{0}(r) \diff r,
\end{equation}
\begin{equation}
\label{eq:energyreg_limit}
\mathcal{E}[n(\tau,\cdot)]+\int_{0}^{\tau}\int_{0}^{R_{b}} \mathds{1}_{rn>0}\, rn \left|{\p_r (\mu+V)}\right|^{2} \diff r \diff t\leq \mathcal{E}[n_{0}],
\end{equation}
\begin{equation}
\label{eq:entropyreg_limit}
\Phi[n(\tau,\cdot)]+\int_{0}^{\tau}\int_{0}^{R_{b}}\left(\gamma r n^{\gamma-1}|\p_{r}n|^{2}+\delta r |\p_{rr}n|^{2}+\delta\f{|\p_{r}n|^{2}}{r} + r\, \p_r n \, \p_r V \right) \diff r \diff t\leq \Phi[n_{0}],
\end{equation}
where energy and entropy are defined as follows:
$$
\mathcal{E}[n] = \int_{0}^{R_{b}}r\,\left(\frac{n^{\gamma+1}}{\gamma+1}+\frac{\delta}{2}\left| {\p_r n} \right|^{2}+n\,V\right) \diff r, \qquad \Phi[n] = \int_{0}^{R_{b}} r\, \phi(n) \diff r,
$$
and $\phi(n) = n \, (\log(n) - 1) + 1$. Equations \eqref{eq:cons_mass_limit}--\eqref{eq:entropyreg_limit} provide the basic a priori estimates. Moreover, we construct Hölder continuous solutions; there is a constant $C$, such that for all $r, r_1, r_2 \in[0,R_{b}]$, $t, t_1, t_2 \in [0,\infty)$
\begin{equation}\label{eq:Holder_in_space_limit}
    |r_1\, n(t,r_1)-r_2\, n(t,r_2)|\le C|r_1-r_2|^{1/2},
\end{equation}
\begin{equation}\label{eq:Holder_in_time_limit}
    |r\,(n(t_{2},r)-n(t_{1},r))|\le C|t_{2}-t_{1}|^{1/8}.
\end{equation}
\end{remark}

\subsection{Regularized system}

We consider the existence of a regularized system, which reads:
\begin{align}
\p_t (r+\eps)n_{\eps} - \p_r \left((r+\eps)B_{\eps}(n_{\eps}) \, {\p_r(\mu_{\eps}+V)} \right)=0,\quad \text{in}\quad &(0,+\infty)\times I_{R_{b}},\label{eq:CHRE1}\\
\mu_{\eps} =n_{\eps}^{\gamma}-\f{\delta}{r+\eps}{\p_r}\left((r+\eps) {\p_r n_{\eps}}\right),\quad \text{in}\quad &(0,+\infty)\times I_{R_{b}} \label{eq:CHRE2},    
\end{align}
where 
\begin{equation}
\label{regmob}
B_{\eps}(n)=
\begin{cases}
\eps\quad\text{for $n\le\eps$,}\\
n\quad\text{otherwise}.
\end{cases}
\end{equation}

\noindent We impose Neumann boundary conditions 
\begin{equation}
\label{Neumann_regularized}
\f{\p n_{\eps}}{\p r}\Big|_{r=0}=\f{\p n_{\eps}}{\p r}\Big|_{r=R_{b}}=B_{\eps}(n_{\eps})\f{\p(\mu_\eps+V)}{\p r}\Big|_{r=0}=B_{\eps}(n_{\eps})\f{\p(\mu_\eps+V)}{\p r}\Big|_{r=R_{b}}=0. 
\end{equation}

We admit the following theorem of existence, for a result for a similar system we refer to~\cite{Yin2001}, 
\begin{thm}\label{thm:classical_sol_approx_system}
For $\eps>0$ and $T>0$, Problem~\eqref{eq:CHRE1}-\eqref{eq:CHRE2} with boundary conditions~\eqref{Neumann_regularized} and smooth initial condition admits a unique strong solution $n_{\eps}$.
\end{thm}

\begin{remark}
Note that the assumption on the initial condition is stronger than the one asked in~\eqref{eq:IDR}. This means that for the regularized system we need to consider a smooth approximation of the initial condition, for instance $n^{0}_{\eps}=n^{0}\ast \omega_{\eps}$ with $\omega$ a smooth kernel that we send to a dirac mass when $\eps\to 0$. 
\end{remark}

\noindent Next, we prove some conservation properties for the system~\eqref{eq:CHRE1}--\eqref{eq:CHRE2}, see for instance~\cite{elbar-perthame-poulain,perthame_poulain}.
\begin{lemma}[Conservation of mass, energy and entropy]\label{lem:cons_mass_ene_ent}
We define $\phi_\eps$ such that $\phi_{\eps}''(n)=\f{1}{B_{\eps}(n)}$ and $\phi_{\eps}(1) = \phi_{\eps}'(1) = 0$, and
\begin{align*}
\mathcal{E}_{\eps}[n]&:=\int_{0}^{R_{b}}(r+\eps)\left(\frac{n^{\gamma+1}}{\gamma+1}+\frac{\delta}{2}\left| {\p_r n} \right|^{2}+n\,V\right) \diff r,\\
\Phi_{\eps}[n]&:=\int_{0}^{R_{b}}(r+\eps)\phi_{\eps}(n) \diff r.
\end{align*}
 Then, we have
\begin{equation}
\label{eq:cons_mass}
\frac{d}{dt} \int_{0}^{R_{b}} (r+\eps)\, n_{\eps}(t,r) \diff r =0,
\end{equation}
\begin{equation}
\label{eq:energyreg}
\frac{d}{dt}\mathcal{E}_{\eps}[n_\eps]+\int_{0}^{R_{b}}(r+\eps)B_{\eps}(n_\eps)\left|{\p_r (\mu_\eps+V)}\right|^{2} \diff r=0,
\end{equation}
\begin{equation}
\label{eq:entropyreg}
\frac{d}{dt}\Phi_{\eps}[n_\eps]+\int_{0}^{R_{b}}\left(\gamma (r+\eps) n^{\gamma-1}_\eps|\p_{r}n_\eps|^{2}+\delta(r+\eps) |\p_{rr}n_\eps|^{2}+\delta\f{|\p_{r}n_\eps|^{2}}{r+\eps} + (r+\eps)\, \p_r n_\eps \, \p_r V \right) \diff r=0.
\end{equation}
\end{lemma}
\begin{remark}\label{rem:entropy}
The function $\phi_\eps$ is given by an explicit formula
\begin{equation}
\label{eq:entropycases}
\phi_{\eps}(x) = 
\begin{cases}
x \, (\log(\varepsilon) - 1) + 1 +x^2/(2\varepsilon) - \varepsilon/2 & \mbox{ for }\;  x \leq \varepsilon \\
x \, (\log(x) - 1) + 1 & \mbox{ for }\; \varepsilon < x.
\end{cases}
\end{equation}
With $\varepsilon < 1$, it enjoys three properties:
\begin{enumerate}
    \item $\phi_\eps(x) \to \phi(x):= x \, (\log(x) - 1) + 1$ for $x \geq 0$ as $\varepsilon \to 0$,
    \item $\phi_\eps(x) \geq 0$ for all $x \in \R$,
    \item $\phi_\eps'(x) \leq 0$ for $x \leq \varepsilon$,
    \item $\phi_\eps(x) \leq \phi(x) + 1 - \eps/2$ for $x \geq 0$.
\end{enumerate}
The first one is trivial. To see the second one, we observe that the function $x \mapsto x(\log(x) - 1) + 1$ is nonnegative, which implies $\phi_\eps(x) \geq 0$ for $x\geq \eps$. Then, for $x \leq \eps$ we discover
\begin{equation}\label{eq:derivative_of_entropy}
\phi_\eps'(x) = \log(\varepsilon) - 1 + \frac{x}{\varepsilon} \leq 0.
\end{equation}
As $\phi_\eps(\varepsilon)\geq 0$, this implies $\phi_{\eps}(x) \geq 0$ for all $x \in \R$. Then, \eqref{eq:derivative_of_entropy} also implies the third property while the forth follows by estimating $\phi_{\eps}(x) \leq \phi_\eps(0)$ for $x \leq \varepsilon$.
\end{remark}
\begin{proof}[Proof of Lemma \ref{lem:cons_mass_ene_ent}]
Mass conservation \eqref{eq:cons_mass} follows from integrating \eqref{eq:CHRE1} in space and using the boundary conditions~\eqref{Neumann_regularized}.

To see \eqref{eq:energyreg}, we multiply \eqref{eq:CHRE1} by $\mu_\eps + V$, integrate in space and use boundary conditions to obtain: 
$$
\int_0^{R_{b}} { (r+\eps) \, \p_t n_{\eps}} (\mu_\eps + V) \diff r + \int_0^{R_{b}} (r+\eps)B_{\eps}(n_{\eps}) \left|{\p_r(\mu_{\eps}+V)} \right|^2 \diff r=0.
$$
Using \eqref{eq:CHRE2} and integrating by parts, we obtain
$$
\int_0^{R_{b}} {(r+\eps)\, \p_t n_{\eps}} (\mu_\eps + V) \diff r = \frac{d}{dt}\mathcal{E}_{\eps}[n_\eps],
$$
which concludes the proof of \eqref{eq:energyreg}. 

To see \eqref{eq:entropyreg}, we multiply \eqref{eq:CHRE1} by $\phi'_\eps(n_\eps)$ and integrate in space to obtain
$$
\frac{d}{dt}\Phi_{\eps}[n_\eps] + \int_0^{R_{b}}  (r+\eps) \, n'_{\eps}(r) \, {\p_r(\mu_{\eps}+V)} \diff r = 0.
$$

In view of \eqref{eq:entropyreg}, it is sufficient to prove
$$
 \int_0^{R_{b}}  (r+\eps) \, n'_{\eps}(r) \, {\p_r \mu_{\eps}} \diff r =  \gamma \, \int_0^{R_{b}}  (r+\eps) \, |n'_{\eps}(r)|^2 \, n_{\eps}^{\gamma-1} \diff r + \delta\int_0^{R_{b}} \frac{|n'_\eps|^2}{r+\eps} + (r+\eps) |n''_\eps|^2 \diff r.
$$
We have
$$
\int_0^{R_{b}}  (r+\eps) \, n'_{\eps}(r) \, {\p_r \mu_{\eps}} \diff r  = \gamma \, \int_0^{R_{b}}  (r+\eps) \, |n'_{\eps}(r)|^2 \, n_{\eps}^{\gamma-1} \diff r - \delta\int_0^{R_{b}}  (r+\eps) \, n'_{\eps}(r) \p_r \left(\frac{1}{r+\eps} \p_r((r+\eps) \p_r n_{\eps}) \right) \diff r.
$$
In the second part, we can integrate by parts (using Neumann's boundary conditions) 
\begin{align*}
- \int_0^{R_{b}}  (r+\eps) \, &n'_{\eps}(r) \p_r \left(\frac{1}{r+\eps} \p_r((r+\eps) \p_r n_{\eps}) \right) \diff r\\ 
&=
 \int_0^{R_{b}}   n'_{\eps}(r) \frac{1}{r+\eps} \p_r((r+\eps) \p_r n_{\eps})  \diff r
 + \int_0^{R_{b}}  (r+\eps) \, n''_{\eps}(r) \frac{1}{r+\eps} \p_r((r+\eps) \p_r n_{\eps})  \diff r \\
 &=   \int_0^{R_{b}} \frac{|n'_\eps|^2}{r+\eps} + (r+\eps) |n''_\eps|^2 \diff r + 2 \int_0^{R_{b}} n'_{\eps}(r) \, n''_{\eps}(r) \diff r.
\end{align*}
The last term vanishes thanks to boundary conditions:
$$
2 \int_0^{R_{b}} n'_{\eps}(r) \, n''_{\eps}(r) \diff r = 
\int_0^{R_{b}} \p_r | n'_\eps(r)|^2 \diff r = 0
$$
and this concludes the proof.
\end{proof}

\noindent From Lemma \ref{lem:cons_mass_ene_ent}, we may deduce uniform bounds (in $\eps$) for the solutions $n_{\eps}$ as follows
\begin{proposition}
\label{prop:apriori}
Let $T>0$. The following sequences are uniformly bounded with respect to $\eps>0$:
\begin{enumerate}[label=(A\arabic*)]
\item \label{item_estim1} $\{\sqrt{r+\eps} \, \p_r n_\eps \}$ in $L^{\infty}((0,\infty); L^2(I_{R_{b}}))$,
\item \label{item_estim2} $\{\sqrt{r+\eps}\, n_{\eps} \}$ in $L^{\infty}((0,\infty) \times I_{R_{b}})$,
\item \label{item_estim3} $\{\sqrt{(r+\eps)\,B_{\eps}(n_{\eps})} \,\p_{r}(\mu_{\eps}+V)\}$ in $L^2((0,\infty)\times I_{R_{b}})$,
\item \label{item_estim4} $\{\sqrt{r+\eps}\, \p_{rr} n_{\eps}\}$ and $\left\{\frac{\p_{r} n_{\eps}}{\sqrt{r+\eps}}\right\}$ in $L^2((0,T)\times I_{R_{b}})$,
\item \label{item_estim4+1/2} $\{\Phi_\eps(n_\eps)\}$ in $L^\infty(0,T)$,
\item \label{item_estim5} $\{(r+\eps) \, \p_t n_\eps \}$ in $L^{2}((0,\infty);H^{-1}(I_{R_{b}}))$,
\end{enumerate}
where the estimates \ref{item_estim1}--\ref{item_estim3} and \ref{item_estim5} depend only on the initial energy $\mathcal{E}(n_0)$. Moreover, there is a constant $C$, independent of $\eps$, such that for all $r, r_1, r_2 \in[0,R_{b}]$, $t, t_1, t_2 \in [0,\infty)$
\begin{equation}\label{eq:Holder_in_space}
    |(r_1+\eps) n_{\eps}(r_1,t)-(r_2+\eps) n_{\eps}(r_2,t)|\le C|r_1-r_2|^{1/2},
\end{equation}
\begin{equation}\label{eq:Holder_in_time}
    |(r+\eps)(n_{\eps}(t_{2},r)-n_{\eps}(t_{1},r))|\le C|t_{2}-t_{1}|^{1/8}.
\end{equation}
In fact, the constant $C$ depends only on initial energy $\mathcal{E}(n_0)$.
\end{proposition}

\begin{proof}[Proof of Proposition \ref{prop:apriori}]

We divide the reasoning into a few steps.

\noindent \textit{Step 1: Estimates \ref{item_estim1}--\ref{item_estim2}}. First, from \eqref{eq:energyreg} we deduce \ref{item_estim1}. For estimate~\ref{item_estim2} we adapt the method from~\cite{Yin2001}. For any $\rho\in(0,R_{b})$,
\begin{align*}
&\frac{R_{b}^{2}+2\eps R_{b}}{2}n_{\eps}(t,\rho)-\int_{0}^{R_{b}}(z+\eps) \,n_{\eps}(t,z) \diff z=\int_{0}^{R_{b}}(z+\eps)\, [n(t,\rho)-n(t,z)] \diff z\\
&=\int_{0}^{R_{b}}\int_{z}^{\rho}(z+\eps) \, \p_{r}n_{\eps}(t,r)\diff r \diff z\\
&=\int_{0}^{\rho}\int_{z}^{\rho}(z+\eps) \, \p_{r}n_{\eps}(t,r)\diff r \diff z+\int_{\rho}^{R_{b}}\int_{z}^{\rho}(z+\eps) \, \p_{r}n_{\eps}(t,r) \diff r \diff z \\
&=\int_{0}^{\rho}\int_{0}^{r}(z+\eps) \, \p_{r}n_{\eps}(t,r) \diff z \diff r+\int_{\rho}^{R_{b}}\int_{r}^{R_{b}}(z+\eps) \, \p_{r}n_{\eps}(t,r) \diff z \diff r\\
&=\int_{0}^{\rho}\left(\f{r^{2}}{2}+\eps\, r\right)\p_{r}n_{\eps}(t,r)\diff r+\int_{\rho}^{R_{b}}\left[\f{1}{2}(R_{b}^{2}-r^{2})+\eps(R_{b}-r)\right]\p_{r}n_{\eps}(t,r) \diff r\\
&\le R_{b}\int_{0}^{\rho}(r+\eps)|\p_{r}n_{\eps}(t,r)| \diff r+2R_{b}^{2}\int_{\rho}^{R_{b}}|\p_{r}n_{\eps}(r,t)| \diff r.
\end{align*}
Multiplying the previous inequality by $2(\rho+\eps)^{1/2}$ yields 
\begin{align*}
&\left|(R_{b}^{2}+2\eps R_{b})(\rho+\eps)^{1/2}n_{\eps}(t,\rho)-2(\rho+\eps)^{1/2}\int_{0}^{R_{b}}(z+\eps)n_{\eps}(t,z) \diff z\right| \leq \\
& \qquad \qquad \le 2(\rho+\eps)^{1/2}R_{b}\int_{0}^{\rho}(r+\eps)|\p_{r}n_{\eps}(t,r)| \diff r+4R_{b}^{2}\int_{\rho}^{R_{b}}(r+\eps)^{1/2}|\p_{r}n_{\eps}(t,r)| \diff r \\
& \qquad \qquad \le C(R_{b})\left(\int_{0}^{R_{b}}(r+\eps)\left|\p_{r}n_{\eps}(t,r)\right|^{2} \diff r\right)^{1/2}.
\end{align*}

\noindent Thanks to the conservation of mass \eqref{eq:cons_mass}, we obtain \ref{item_estim2}.\\

\noindent \textit{Step 2: Estimates \ref{item_estim3}--\ref{item_estim4+1/2}.} The bound \ref{item_estim3} follows from the conservation of energy \eqref{eq:energyreg}. To see \ref{item_estim4} and \ref{item_estim4+1/2}, we want to use the conservation of entropy \eqref{eq:entropyreg}, but this has to be done carefully, as the term $(r+\eps)\, \p_r n_\eps \, \p_r V$ can be negative. Therefore, we fix $T>0$, consider $\phi_\eps$ as in Remark \ref{rem:entropy} and integrate \eqref{eq:entropyreg} on $(0,T)$ to deduce
$$
\Phi_{\eps}(n_\eps(T,\cdot)) +\delta \int_0^T \int_0^{R_{b}} (r+\eps) |\p_{rr}n_\eps|^{2}+\f{|\p_{r}n_\eps|^{2}}{r+\eps} \diff r \diff t \leq \Phi_{\eps}(n_0) + \int_0^T \int_0^{R_{b}} (r+\eps)\, \p_r n_\eps \, \p_r V \diff r \diff t.
$$
The last term can be easily bounded (using estimates \ref{item_estim1}-\ref{item_estim2}) by a constant depending on $T$. The conclusion follows from $\Phi_{\eps}(n_{\eps}(T,\cdot)) \geq 0$ and $\Phi_{\eps}(n_0)$ can be bounded in terms of $\Phi(n_0)$, cf. (4) in Remark \ref{rem:entropy}.\\

\noindent \textit{Step 3: Estimate \ref{item_estim5}.} Let $\chi\in L^{2}(0,T;H^{1}(I_{R_{b}}))$. We multiply \eqref{eq:CHRE1} by $\chi$ and integrate with respect to $r$ between 0 and $R_{b}$. Using an integration by parts and Neumann boundary conditions, we obtain

\begin{align*}
    \int_{0}^\infty \int_0^{R_{b}} (r+\eps)\p_{t}n_{\eps}&\chi \diff r \diff t =-\int_{0}^{\infty} \int_0^{R_{b}} (r+\eps)B_{\eps}(n_{\eps})\p_{r}(\mu_{\eps}+V)\p_{r}\chi \diff r \diff t \\
    &=-\int_{0}^\infty \int_0^{R_{b}} \sqrt{(r+\eps)\, B_{\eps}(n_{\eps})}\, \sqrt{(r+\eps)\, B_{\eps}(n_{\eps})} \, \p_{r}(\mu_{\eps}+V) \,\p_{r}\chi \diff r \diff t \\
    &\le \norm{\sqrt{(r+\eps)\, B_{\eps}(n_{\eps})}}_{\infty}\, \norm{\sqrt{(r+\eps)\, B_{\eps}(n_{\eps})}\,\p_{r}(\mu_{\eps}+V)}_{2} \, \norm{\p_{r}\chi}_{2}.\phantom{\int_{0}^T}
\end{align*}
where the norms are taken over $(0,\infty)\times I_{R_{b}}$. The conclusion follows.\\

\noindent \textit{Step 4: Hölder estimate in space \eqref{eq:Holder_in_space}.} By differentiation
$$
(r_2 + \varepsilon) \, n(t,r_2) - (r_1 + \varepsilon) \, n(t,r_1) = \int_{r_1}^{r_2} (r+\varepsilon) \, \p_r n(t,r) \diff r + \int_{r_1}^{r_2}   n(t,r) \diff r.
$$
For the first term, we have
$$
\int_{r_1}^{r_2} (r+\varepsilon) \, \p_r n_\eps(t,r) \diff r \leq 
\left(\int_{r_1}^{r_2} (r+\varepsilon) \diff r \right)^{1/2} \,
\|  \sqrt{r+\varepsilon} \, \p_r n_\eps(t,\cdot) \|_2 \leq C \, |r_1 - r_2|^{1/2}
$$
due to \ref{item_estim1}. For the second term, we compute, using \ref{item_estim2},
\begin{multline*}
\int_{r_1}^{r_2} n_\eps(t,r) \diff r = \int_{r_1}^{r_2} n_\eps(t,r) \frac{\sqrt{r+\eps}}{\sqrt{r+\eps}} \diff r \leq  \|\sqrt{r+\eps}\, n_\eps  \|_{\infty} \int_{r_1}^{r_2} \frac{1}{\sqrt{r+\eps}} \diff r \leq \\ \leq C\, |\sqrt{r_2 + \eps} - \sqrt{r_1 + \eps}| \leq C \, |r_1 - r_2|^{1/2}.
\end{multline*}

\noindent \textit{Step 5: Hölder estimate in time \eqref{eq:Holder_in_time}.} The idea is to deduce the regularity in time from the regularity in space.
We extend the function $r \mapsto n_\eps\,r$ for $r < 0$ with a constant to preserve continuity. We consider $\eta_{\nu}$ to be a usual one-dimensional mollifier in the spatial variable $r$ where $\nu$ will be chosen later in terms of $|t_2 - t_1|$. Mollifying \eqref{eq:CHRE1} with $\eta_{\nu}$ and integrating in time (from $t_1$ to $t_2$) we obtain
\begin{equation}\label{eq:mollified_in_space_eq}
((r+\eps)n_{\eps})\ast \eta_{\nu}(t_2,r) - ((r+\eps)n_{\eps})\ast \eta_{\nu}(t_1,r)   = \int_{t_1}^{t_2} \p_r \eta_{\nu} \ast \left((r+\eps)B_{\eps}(n_{\eps}) \, {\p_r(\mu_{\eps}+V)} \right) \diff t.
\end{equation}
First, we estimate (RHS). We notice that Young's convolutional inequality and Hölder's inequality are implying for fixed $t \in [t_1, t_2]$
\begin{multline*}
\| \p_r \eta_{\nu} \ast \left((r+\eps)B_{\eps}(n_{\eps}) \, {\p_r(\mu_{\eps}+V)} \right) \|_{\infty} \leq \\ \leq
\|\p_r \eta_{\nu} \|_{2} \, \| \sqrt{(r+\eps)B_{\eps}(n_{\eps})} \|_{\infty} \, 
\| \sqrt{(r+\eps)B_{\eps}(n_{\eps})} \, {\p_r(\mu_{\eps}+V)}  \|_{2}.
\end{multline*}
By the definition of a mollifier,
$$
\|\p_r \eta_{\nu} \|_{2} = \frac{1}{\nu^2} \left|\int_{\mathbb{R}} (\eta')^2\left(\frac{r}{\nu} \right) \diff r \right|^{1/2} \leq \frac{C}{\nu^{3/2}}.
$$
Therefore, applying \ref{item_estim2}, \ref{item_estim3} and Hölder's inequality in time, we deduce
\begin{equation}\label{eq:term_with_der_r_mollified_est}
\int_{t_1}^{t_2} \| \p_r \eta_{\nu} \ast \left((r+\eps)B_{\eps}(n_{\eps}) \, {\p_r(\mu_{\eps}+V)} \right) \|_{\infty} \diff t \leq C \, \frac{|t_2-t_1|^{1/2}}{\nu^{3/2}}.
\end{equation}
To conclude the proof, we need to estimate \eqref{eq:mollified_in_space_eq} using \eqref{eq:Holder_in_space} from Step 4, we get
\begin{equation}\label{eq:correction_moll_nomoll}
\begin{split}
|((r+\eps)n_{\eps})\ast \eta_{\nu}(t_1,r) -& 
((r+\eps)n_{\eps})(t_1,r)|
\leq
\\
&\leq
\int_{\R} |(r+y+\eps)n_{\eps}(t_1,r+y) - (r+\eps) n_{\eps}(t_1,r) | \eta_{\nu}(y) \diff y 
\\
&\leq C\int_{\R} |y|^{1/2}\, \eta_{\nu}(y) \diff y \leq C\,\nu^{1/2},
\end{split}
\end{equation}
where we used that on the support of $\eta_{\nu}$ we have $|y|\leq \nu$. Exactly the same estimate holds if we replace $t_1$ with $t_2$. Combining \eqref{eq:mollified_in_space_eq}, \eqref{eq:term_with_der_r_mollified_est} and \eqref{eq:correction_moll_nomoll} we obtain
$$
|(r+\eps) n_{\eps}(t_2,r) - (r+\eps) n_{\eps}(t_1,r)| \leq C\, \frac{|t_2 - t_1|^{1/2}}{\nu^{3/2}} + C\, \nu^{1/2}.
$$
We choose $\nu = |t_2 - t_1|^{1/4}$ and this concludes the proof.
\end{proof}

\begin{remark}\label{rem:Linf_est_sqrtrn}
In the above proof, Step 1 shows more generally that when $n(t,r):[0,\infty) \times [0,R_b] \to \R$ satisfies $\int_{0}^{R_b} r\,n(t,r) \diff r = m$ and $\sqrt{r} \, \p_r n \in L^{\infty}(0,T; L^2(0, R_b))$ then
$$
\left|\sqrt{r} \, n(t,r)\right| \leq C\left(R_b, m, 
\|\sqrt{r} \, \p_r n \|_{L^{\infty}(0,T; L^2(0, R_b))}  \right).
$$
\end{remark}

\subsection{Proof of Theorem~\ref{thm:finalthm1} (existence part)}\label{subsect:eps_to_zero}

We are concerned with the first part of Theorem~\ref{thm:finalthm1} \ie the convergence $\eps\to 0$ of the approximation scheme. 

\begin{proof}[Proof of Theorem~\ref{thm:finalthm1} (existence)]
The proof is divided into several steps.\\

\noindent {\it Step 1: Compactness}. By the estimates in Proposition \ref{prop:apriori}, the Banach-Alaoglu and Arzela-Ascoli theorems, we can extract a subsequence such that, for some $\xi \in L^2((0,T)\times I_{R_b})$,
\begin{enumerate}[label=(C\arabic*)]
    \item \label{item_conv_1}$(r+\varepsilon)\,n_\eps \to r\,n$ uniformly in $C([0,T]\times I_{R_b})$,
    \item $(r+\eps) \, \p_t n_\eps \rightharpoonup r\, \p_t n$ in $L^2(0,T; H^{-1}(I_{R_b}))$,
    \item\label{item_conv_3} $\sqrt{(r+\varepsilon)\, B_\eps(n_\eps)} \p_r(\mu_\eps + V) \rightharpoonup \xi \mbox{ in } L^2((0,T)\times I_{R_b})$,
    \item \label{item_conv4}$\sqrt{r+\eps}\, \p_{rr} n_{\eps} \rightharpoonup \sqrt{r}\, \p_{rr} n$ and $\frac{\p_{r} n_{\eps}}{\sqrt{r+\eps}} \rightharpoonup \frac{\p_{r} n}{\sqrt{r}}$ in $L^2((0,T)\times I_{R_b})$.
\end{enumerate}

\noindent \textit{Step 2: Nonnegativity of $n$.} The plan is to obtain a contradiction with the uniform estimate of the entropy. For $\alpha >0$, we define the sets
\begin{equation*}
  V_{\alpha,\eps}=\{(t,r)\in (0,T)\times I_{R_{b}}: \,  n_{\eps}(t,r)\le-\alpha, \, r \geq \alpha \},
\end{equation*}
\begin{equation*}
  V_{\alpha,0}=\{(t,r)\in (0,T)\times I_{R_{b}}: \,  n(t,r)\le-\alpha, \, r \geq \alpha \}.
\end{equation*}
By Remark \ref{rem:entropy} (nonnegativity of $\phi_\eps$) and \ref{item_estim4+1/2} in Lemma \ref{prop:apriori}  there is a constant such that
\begin{equation*}
 \int_{V_{\alpha,\eps}} (r+\eps)\, \phi_\eps(n_\eps) \diff r \diff t \leq   \int_{(0,T)\times I_{R_{b}}} (r+\eps)\, \phi_\eps(n_\eps) \diff r \diff t \le C(T).
\end{equation*}
For $n_\eps \leq -\alpha$, we have $0 \leq \phi_\eps(-\alpha) \leq \phi_\eps(n_\eps)$ ((3) in Remark \ref{rem:entropy}) so that
$$
\left(-\alpha (\log(\eps) - 1) + 1 + \alpha^2/(2\varepsilon) - \eps/2 \right)\, \int_{V_{\alpha,\varepsilon}} (r+\varepsilon) \diff r \diff t\leq C(T).
$$
Sending $\varepsilon \to 0$ and using uniform convergence of $n_{\eps} \to n$ for $r \geq \alpha > 0$ we discover that
$$
\int_{V_{\alpha, 0}} r \diff r \diff t = \lim_{\eps \to 0} \int_{V_{\alpha, \eps}} (r+\eps) \diff r \diff t = 0
$$
using, from measure theory, that on a measure space $(X,\mu)$ if $f_n, f: X \to \R$ and $f_n \to f$ in $L^1(X,\mu)$ then for $\alpha \in \R$ we have $\int_{f_n < \alpha} \diff \mu \to \int_{f < \alpha} \diff \mu$ as $n \to \infty$. This means that $V_{\alpha,0}$ is a null set for each $\alpha > 0$, concluding the proof.\\

\noindent \textit{Step 3: Identification of the limit $(r+\eps)B_{\eps}(n_{\eps})\p_{r}(\mu_{\eps}+V)$.} The last difficulty is to pass to the limit in $\int_{0}^{T}\int_{0}^{R_{b}}(r+\eps)B_{\eps}(n_{\eps})\p_{r}(\mu_{\eps}+V)\p_{r}\varphi\diff r\diff t$. Indeed, since the mobility is degenerate it is not clear that we can identify the derivative of the potential $\p_{r}\mu$ in the limit. However, due to the uniform convergence of $(r+\eps)n_{\eps}$ and the nonnegativity of $n$ we can conclude. By \ref{item_conv_3} and the uniform convergence of $\sqrt{(r+\varepsilon)\, B_\eps(n_\eps)}$, we have
\begin{equation}\label{eq:weak_conv_dr_mu_times_nr}
(r+\varepsilon)\, B_\eps(n_\eps) \p_r(\mu_\eps + V) \rightharpoonup \sqrt{r\,n}\,\xi = \begin{cases}
\sqrt{r\,n}\,\xi &\mbox{ when } rn > 0\\
0 &\mbox{ when } rn = 0
\end{cases} \mbox{ in } L^2((0,T)\times I_{R_b}).
\end{equation}
We first claim that
\begin{equation}\label{eq:xi_char}
\xi(t,r) = 
\sqrt{r n} \, \p_r(\mu+V)  \mbox{ when } rn > 0.
\end{equation}
We introduce the family of open sets
$$
\{(t,r): r\,n(t,r) > 0 \} = \cup_{\nu > 0} P_{\nu}, \qquad P_{\nu}=\{(t,r): rn(r,t)>\nu,\,  r>\nu\},
$$
so that it is sufficient to identify the limit $\xi$ in $P_\nu$ for fixed $\nu>0$.\\ 

\noindent Because of the uniform convergence we know that for every $\eps<\eps(\nu)$ for $\eps(\nu)$ small enough, 

\begin{equation*}
    (r+\eps)B_{\eps}(n_{\eps}(r,t))\ge \f{\nu}{2}, \quad (r,t)\in P_{\nu}. 
\end{equation*}
Therefore, the estimate~\ref{item_estim3} implies
\begin{equation*}
    \norm{\p_{r}(\mu_{\eps}+V)}_{L^{2}(P_{\nu})}\le\f{C}{\nu^{1/2}}.
\end{equation*}
As $\p_rV$ is uniformly bounded, we deduce that

\begin{equation*}
    \norm{\p_{r}\mu_{\eps}}_{L^{2}(P_{\nu})}\le\f{C}{\nu^{1/2}}. 
\end{equation*}
By definition of $\mu_\eps$
$$
\p_{r}\mu_{\eps}=\gamma n^{\gamma-1}_{\eps}\p_{r}n_{\eps}-\delta\p_{r}\left(\f{1}{r+\eps}\p_{r}((r+\eps)\p_{r}n_{\eps})\right).
$$
For the first term of (RHS), we use the strong convergence~\ref{item_conv_1} that yields a uniform convergence of $n_{\eps}$ in the zone $\{(r,t): r>\nu\}$. Then, because $P_{\nu}\subset\{(r,t): r>\nu\}$ we obtain   

\begin{equation*}
\gamma n^{\gamma-1}_{\eps}\to \gamma n^{\gamma-1}\quad \text{uniformly in $L^{\infty}(P_{\nu})$}.    
\end{equation*}
Combined with the weak convergence provided by estimate~\ref{item_estim1} in $P_{\nu}$ we obtain that up to a subsequence,

\begin{equation*}
\gamma n^{\gamma-1}_{\eps}\p_{r}n_{\eps}\rightharpoonup \gamma n^{\gamma-1}\p_{r}n\quad\text{weakly in $L^{2}(P_{\nu})$}.    
\end{equation*}

\noindent Then we combine the $L^{2}(P_{\nu})$ bound on $\gamma n^{\gamma-1}_{\eps}\p_{r}n_{\eps}$ with the $L^{2}(P_{\nu})$ estimate of $\p_{r}\mu_{\eps}$.  We obtain an $L^{2}(P_{\nu})$ bound on the second term on the right-hand side. Together with estimates~\ref{item_estim1}-\ref{item_estim2}-\ref{item_estim3}-\ref{item_estim4} we obtain the weak convergence up to a subsequence
\begin{equation*}
 \p_{r}\left(\f{1}{r+\eps}\p_{r}((r+\eps)\p_{r}n_{\eps})\right)\rightharpoonup \p_{r}\left(\f{1}{r}\p_{r}(r\p_{r}n)\right)\quad\text{weakly in $L^{2}(P_{\nu})$}. 
\end{equation*}
Finally, we obtain
\begin{equation}
\label{eq:weakconvP2}
\p_{r}\mu_{\eps}\rightharpoonup \p_{r}\mu=\gamma n^{\gamma-1}\p_{r}n-\delta\p_{r}\left(\f{1}{r}\p_{r}(r\p_{r}n)\right)\quad \text{weakly in $L^{2}(P_{\nu})$}.    
\end{equation}
Using uniform convergence, we conclude the proof of \eqref{eq:xi_char}. Finally, \eqref{eq:weak_conv_dr_mu_times_nr} and \eqref{eq:xi_char} implies
$$
\int_{0}^{T}\int_{0}^{R_{b}}(r+\eps)B_{\eps}(n_{\eps})\p_{r}(\mu_{\eps}+V)\p_{r}\varphi\diff r\diff t \to \int_{rn>0}rn\p_{r}(\mu+V)\p_{r}\varphi\diff r\diff t.
$$

\noindent {\it Step 4: existence of a weak solution.} Steps 1-3 show that $n$ satisfies the condition of Definition \ref{def:weak}.
\\

\noindent {\it Step 5: Properties 
\eqref{eq:cons_mass_limit}--\eqref{eq:Holder_in_time_limit} from Remark \ref{rem:more_prop_weak_sol}.} First, properties \eqref{eq:cons_mass_limit}, \eqref{eq:Holder_in_space_limit} and \eqref{eq:Holder_in_time_limit} follow from uniform convergence \ref{item_conv_1} and estimates \eqref{eq:Holder_in_space}-\eqref{eq:Holder_in_time} for $r\, n_\eps$. To see \eqref{eq:energyreg_limit}, we notice that weak lower semicontinuity of $L^2$ norm implies
$$
\mathcal{E}[n(\tau,\cdot)]+\int_0^\tau\int_{0}^{R_{b}} |\xi(t,r)|^2 \diff r \diff t \leq \mathcal{E}[n_0].
$$
By \eqref{eq:xi_char}, the integral on the (LHS) can be estimated from below by 
$$
\mathcal{E}[n(\tau,\cdot)]+\int_0^\tau\int_{0}^{R_{b}} \mathds{1}_{rn>0}\, rn \left|{\p_r (\mu+V)}\right|^{2} \diff r \diff t \leq \mathcal{E}[n_0].
$$
To see \eqref{eq:entropyreg_limit}, it is sufficient to prove
$$
\Phi(n(t,\cdot)) \leq \liminf_{\varepsilon \to 0} \Phi_\eps(n_\eps(t,\cdot)).
$$
Let $\delta > 0$. By nonnegativity of $\phi_{\varepsilon}$ we estimate
$$
\liminf_{\varepsilon \to 0} \Phi_\eps(n_\eps(t,\cdot)) \geq \liminf_{\varepsilon \to 0} \int_{r \geq \delta} \phi_\eps(n_\eps) (r+\eps) \diff r = \int_{r \geq \delta} \phi(n) \, r \diff r,
$$
because on the set $\{r \geq \delta\}$, we have uniform convergence $n_\eps \to n$. As $\phi(n) \geq 0$, we can send $\delta \to 0$ by monotone convergence and conclude the proof.
\\ 

\noindent {\it Step 6: Neumann boundary condition $n'(t,R_{b})=0$.} First, if $\varphi, \phi \in C^1[a,b]\cap H^2(a,b)$ we have (via approximation)
\begin{equation}\label{eq:integration_by_parts_smooth}
\int_{a}^b \varphi'(r) \, \phi'(r) + \varphi''(r) \, \phi(r) \diff r = \varphi'(b) \, \phi(b) - \varphi'(a)\, \phi(a).
\end{equation} 
Let $\phi$ be a smooth function with $\phi(R_0) = 0$ and $\phi(R_b) = 1$ for some $R_0 \in (0,R_b)$. We know from estimates~\ref{item_conv_1}-\ref{item_conv4} that $n\in L^{2}(0,T;H^{2}(R_0, R_b))$. Let $t$ be such that $r \mapsto n(t,r) \in H^2(R_0, R_b)$. Applying ~\eqref{eq:integration_by_parts_smooth} with $\varphi(r) = n_\eps(t,r)$, we deduce, thanks to the Neumann boundary condition $n_\eps'(t,R_b) = 0$, that
$$
\int_{R_0}^{R_b} n_\eps'(t,r) \, \phi'(r) + n_\eps''(t,r) \, \phi(r) \diff r = 0.
$$
Multiplying by a smooth test function $\eta(t)$, we have
$$
\int_0^T \int_{R_0}^{R_b} \eta(t) \left( n_\eps'(t,r) \, \phi'(r) + n_\eps''(t,r) \, \phi(r)\right) \diff  r \diff t = 0.
$$
Passing to the weak limit $\eps \to 0$ and using that $\eta$ is arbitrary we conclude
$$
\int_{R_0}^{R_{b}} n'(t,r) \, \phi'(r) + n''(t,r) \, \phi(r) \diff r = 0, \quad \text{for a.e. } t>0.
$$
 As $n(t,\cdot) \in H^2(R_0,R_b)$, we can apply~\eqref{eq:integration_by_parts_smooth} again and deduce
$$
\int_{R_0}^{R_{b}} n'(t,r) \, \phi'(r) + n''(t,r) \, \phi(r) \diff r = n'(t,R_b),
$$
which finally proves $n'(t,R_b) = 0$.
\end{proof}

\subsection{Proof of Theorem~\ref{thm:finalthm1} (Long term asymptotics)}

With global solutions at hand, we can study the long term behaviour. For that purpose, we fix $k,T$, $k\ge T$ and define $n_{k}(t,x)=n(t+k,x)$, $\mu_{k}(t,x)=\mu(t+k,x)$. Consider the solution $n$ in the interval $(-T+k, T+k)$, it satisfies 
\begin{equation*}
\int_{-T+k}^{T+k}r\langle\p_{t}n,\varphi\rangle_{H^{-1},H^{1}} \diff t +\int_{-T+k}^{T+k} \int_{0}^{R_{b}}\mathds{1}_{rn>0}\, r\, n\, \p_{r}(\mu+V)\p_{r}\varphi \diff r \diff t=0,
\end{equation*} 
and a change of variables yields 
\begin{equation}
\label{eq:solnk}
\int_{-T}^{T}r\langle\p_{t}n_{k},\varphi\rangle_{H^{-1},H^{1}} \diff t +\int_{-T}^{T} \int_{0}^{R_{b}}\mathds{1}_{rn_{k}>0}\, r\, n_{k}\, \p_{r}(\mu_{k}+V)\p_{r}\varphi \diff r \diff t=0.
\end{equation} 
We also recall the Neumann boundary condition $n_{k}'(t,R_{b})=0$ and the conservation of mass $\int_{0}^{R_{b}}rn_{k}\diff r=\int_{0}^{R_{b}}rn_{0}\diff r$. We want to pass to the limit $k\to\infty$ in this equation and prove the

\begin{proposition}
Let $(n,\mu)$ be a weak solution of~\eqref{eq:CHR1}-\eqref{eq:CHR2}. Then, we can extract a subsequence, still denoted by the index $k$, of $(n_{k},\mu_{k})$ such that $\sqrt{r}n_{k}\to \sqrt{r}n_{\infty}$ strongly in $L^{\infty}((-T,T) \times I_{R_{b}})$ and $\sqrt{rn_{k}}\p_{r}(\mu_{k}+V)\rightharpoonup\sqrt{rn}\p_{r}(\mu_{\infty}+V)$ weakly in $L^{2}((-T,T) \times I_{R_{b}}\setminus\{rn=0\})$. We have $n_{\infty}\in C^{1}(\R \times \overline{B}_{R_{b}})$ and the relations
\\
\begin{equation}
\label{eq:stationarystates1}
    rn_{\infty}\p_{r}(\mu_{\infty}+V)=0, \qquad \mu_\infty= n^{\gamma}_\infty - \f{\delta}{r}\p_{r}(r\p_{r}n_\infty),
\end{equation}
with the Neumann boundary conditions
\begin{equation*}
\f{\p n_{\infty}}{\p r}\Big|_{r=0}=\f{\p n_{\infty}}{\p r}\Big|_{r=R_{b}}=0. 
\end{equation*}
The mass $\int_{0}^{R_{b}}rn_{\infty}(t) \diff r$ is constant and equal to the initial mass $\int_{0}^{R_{b}}rn_{0} \diff r$. 
\end{proposition}

\noindent This proposition implies the assertions of Theorem~\ref{thm:finalthm1}.

\begin{proof}

\noindent\textit{Step 1: Bounds coming from the energy.} We claim that the following uniform estimates (with respect to $k$) are true:
\begin{enumerate}[label=(B\arabic*)]
\item \label{item_estim12} $\{\sqrt{r} \, \p_r n_k \}$ in $L^{\infty}((-T,T); L^2(I_{R_{b}}))$,
\item \label{item_estim12+1/2} $\{\sqrt{r} \, n_k \}$ in $L^{\infty}((-T,T) \times I_{R_{b}})$,
\item \label{item_estim42} $\{r \, \p_t n_k \}$ in $L^{2}((-T,T);H^{-1}(I_{R_{b}}))$,
\item \label{item_estim32} $|r_{2}n_{k}(t_{2},r_{2})-r_{1}n_{k}(t_{1},r_{1})|\le C(|t_{2}-t_{1}|^{1/8}+|r_{2}-r_{1}|^{1/2})$,
\item \label{item_estim22}$L_{k}(T) := \int_{-T}^{T}\int_{0}^{R_{b}}\mathds{1}_{rn_{k}>0}\,rn_{k}|\p_{r}(\mu_{k}+V)|^{2} \xrightarrow[k\to +\infty]{}0. $
\end{enumerate}

\noindent The energy decay estimate \eqref{eq:energyreg_limit} and assumption $\mathcal{E}[n_0] < \infty$ imply that $\mathcal{E}[n_{k}(t)]$ remains bounded with respect to $k$ for all $k>T$. Therefore, \ref{item_estim12} follows directly from \eqref{eq:energyreg_limit} and then \ref{item_estim12+1/2} follows from Remark \ref{rem:Linf_est_sqrtrn}. As $r\,n_k(t,r)$ is obtained as the pointwise limit of $(r+\varepsilon)\,n_{\eps}(t+k,r)$, estimates~\ref{item_estim42} and~\ref{item_estim32} follow directly from passing to the limit $\varepsilon \to 0$ in \ref{item_estim5} and from~\eqref{eq:Holder_in_space}--\eqref{eq:Holder_in_time} in Proposition~\ref{prop:apriori}. Finally, to see~\ref{item_estim22}, we note that
$$
\int_{0}^{\infty}\int_{0}^{R_{b}}\mathds{1}_{rn>0}\,rn|\p_{r}(\mu+V)|^{2} \diff r \diff t\le\mathcal{E}(n_{0}),
$$
so by change of variables we obtain
\begin{equation*}
L_{k}(T)\le \int_{k-T}^{\infty}\int_{0}^{R_{b}}\mathds{1}_{rn>0}\,rn|\p_{r}(\mu+V)|^{2} \diff r \diff t \xrightarrow[k\to +\infty]{}0.
\end{equation*}
\\

\noindent\textit{Step 2: Bounds coming from the entropy.} We prove now uniform estimates
\begin{enumerate}[label=(C\arabic*)]
\item \label{item_estim13} $\{\sqrt{r} \, \p_{rr} n_k \}$ in $L^{2}((-T,T); L^2(I_{R_{b}}))$,
\item \label{item_estim23}$\{\f{\p_{r}n_{k}}{\sqrt{r}}\}$ in $L^{2}((-T,T); L^2(I_{R_{b}}))$.
\end{enumerate}
To this end, we integrate the entropy relation~\eqref{eq:entropyreg_limit} between $k-T$ and $k+T$ and perform a change of variables to obtain
\begin{multline*}
\int_{-T}^{T}\int_{0}^{R_{b}}\left(\gamma r n_{k}^{\gamma-1}|\p_{r}n_{k}|^{2}+\delta r |\p_{rr}n_{k}|^{2}+\delta\f{|\p_{r}n_{k}|^{2}}{r}\right)\diff r \diff t \leq \\ \leq \Phi[n_k(-T, \cdot)]- \Phi[n_k(T,\cdot)]+ \int_{-T}^{T}\int_{0}^{R_{b}} r\, \p_r n_k \, \p_r V\diff r \diff t.
\end{multline*}
We need to bound the right-hand side. Concerning the entropy term, we recall the inequality $\log n \leq n-1$ valid for $n \geq 0$ so that, by bound~\ref{item_estim12+1/2},
\begin{multline*}
\Phi(n_k(T, \cdot)) = \int_{0}^{R_{b}} r( n_k(T, r)(\log n_k(T,r) -1)+1) \diff r \leq \\ \leq \int_0^{R_b} r\, ((n_k(T,r))^2 + n_k(T,r)) \diff r \leq C\, \|\sqrt{r}\,n_k\|_{\infty} \leq C .
\end{multline*}
The same estimate is satisfied by $\Phi(n_k(T))$. Concerning $\int_{-T}^{T}\int_{0}^{R_{b}} r\, \p_r n_k \, \p_r V\diff r \diff t$, we estimate it using \ref{item_estim12} and uniform bound on $\p_r V$.
Therefore,

\begin{equation*}
\int_{-T}^{T}\int_{0}^{R_{b}}\left(\gamma r n_{k}^{\gamma-1}|\p_{r}n_{k}|^{2}+\delta r |\p_{rr}n_{k}|^{2}+\delta\f{|\p_{r}n_{k}|^{2}}{r}\right)\diff r \diff t \leq  C(T,\mathcal{E}(n_{0})). 
\end{equation*}

\noindent\textit{Step 3: Convergence in equation~\eqref{eq:solnk}.} Reasoning as in the proof of Theorem~\ref{thm:finalthm1} we obtain in the limit $k\to\infty$ 
\begin{equation*}
\int_{-T}^{T}r\langle\p_{t}n_{\infty},\varphi\rangle_{H^{-1},H^{1}} \diff t +\int_{-T}^{T} \int_{0}^{R_{b}}\mathds{1}_{rn_{\infty}>0}\, r\, n_{\infty}\, \p_{r}(\mu_{\infty}+V)\p_{r}\varphi \diff r \diff t=0.
\end{equation*} 
We can show even better, namely that $\p_t n_{\infty} = 0$. Indeed, from the Cauchy-Schwarz inequality we obtain that for every test function $\chi$ compactly supported in $(-T,T)\times(0,R_{b})$, 
\begin{align*}
   \left| \int_{-T}^{T}\int_{0}^{R_{b}}r\p_{t}n_{k}\chi\right|&= \left|\int_{-T}^{T}\int_{0}^{R_{b}}\mathds{1}_{rn_{k}>0}\,rn_{k}\p_{r}(\mu_{k}+V)\p_{r}\chi \right|\\
    &\le C(T,R_{b})\norm{\p_{r}\chi}_{L^{\infty}}\int_{-T}^{T}\int_{0}^{R_{b}}\mathds{1}_{rn_{k}>0}\,rn_{k}|\p_{r}(\mu_{k}+V)|^{2}\xrightarrow[k\to\infty]{} 0.
\end{align*}
where we used \ref{item_estim12+1/2} and \ref{item_estim22}. This means that in the limit, $n_{\infty}$ does not depend on the time variable $t$. 
Then, in the limit,  we obtain that, for every test function $\chi$,

 $$
 \int_{-T}^{T}\int_{0}^{R_{b}}\mathds{1}_{rn_{\infty}>0}\,rn_{\infty}\p_{r}(\mu_{\infty}+V)\p_{r}\chi=0.
 $$

\noindent\textit{Step 4: $n_{\infty}'$ is uniformly continuous and $n_{\infty}$ satisfies Neumann boundary condition $n_{\infty}'(0) = 0$.}
We recall that $n_{\infty}$ does not depend on time. Moreover, the estimate \ref{item_estim13} implies that $n'_{\infty}$ is continuous on $(0,R_b]$. Furthermore, from the estimates~\ref{item_estim13}-\ref{item_estim23}, we obtain the absolute continuity in space of the derivative of $n_{\infty}$. Indeed, for every $r_{1}, r_{2}\in (0,R_{b})$ we obtain

\begin{align*}
(\p_{r}n_{\infty}(r_{2}))^{2}-(\p_{r}n_{\infty}(r_{1}))^{2}&=2\int_{r_{1}}^{r_{2}}\p_{r}n_{\infty}(r) \, \p_{rr}n_{\infty}(r)\diff r\\
&=2\int_{r_{1}}^{r_{2}}\f{\p_{r}n_{\infty}(r)}{\sqrt{r}} \, \sqrt{r} \, \p_{rr}n_{\infty}(r)\diff r\\
&\le 2 \Big(\int_{r_{1}}^{r_{2}}\f{|\p_{r}n_{\infty}(r)|^{2}}{r}\diff r\Big)^{1/2}\Big(\int_{r_{1}}^{r_{2}}r |\p_{rr}n_{\infty}(r)|^{2}\diff r\Big)^{1/2} .
\end{align*}
From this, we deduce that $\p_r n_{\infty}$ is bounded so that by the Sobolev embedding, $n_{\infty}$ is continuous and 
$$
n_{\infty}(r_2) - n_{\infty}(r_1) = \int_{r_1}^{r_2} \p_r n_{\infty}(r) \diff r.
$$
Next, we discover that $(0,R_b] \ni r \mapsto (\p_r n_{\infty}(r))^2$ is uniformly continuous, so that by Lemma \ref{lem:f_UC_from_f2_UC} below, $n_{\infty}'(r)$ 
is uniformly continuous on $(0,R_b]$. Therefore, there is the unique extension of $r \mapsto n_{\infty}'(r)$ to $[0,R_b]$ which is uniformly continuous. Furthermore, in view of 
\begin{equation*}
\label{n'(0)=0}
\int_{0}^{R_{b}}\f{|\p_{r}n_{\infty}|^{2}}{r}\diff r \le C,     
\end{equation*}
this extension has to be obtained by setting $n_{\infty}'(0)=0$.
\\ 

\noindent It remains to prove that $n_{\infty}$ is differentiable (in the classical sense) at $r = 0$ and $n_{\infty}'(0) = 0$. To this end, we write
$$
\left|\frac{n_{\infty}(r) - n_{\infty}(0)}{r} \right| \leq \frac{1}{r}  \int_0^r \left| \p_r n_{\infty}(u)\right| \diff u  \leq \sup_{u \in (0,r]} |\p_r n_{\infty}(u)| \to 0
$$
as $r \to 0$ by uniform continuity which, again,  implies that $n_{\infty}'(0)$ exists and $n_{\infty}'(0) = 0$.
\\

\noindent\textit{Step 5: Neumann boundary condition $n_{\infty}'(R_{b})=0$.} The proof is similar to Step~6 in Section~\ref{subsect:eps_to_zero}. For a fixed $k \in \N$, there is a set of times $\mathcal{N}_k \subset (0,T)$ of full measure such that, when $t \in \mathcal{N}_k$, we have $n_k'(t,R_b) = 0$ and $n_k(t,\cdot) \in H^2(R_0, R_b)$. Let $\mathcal{N} = \cap_{k \in \N} \mathcal{N}_k$, which is again the set of full measure. For $t \in \mathcal{N}$ and $\phi$ as in Step~6 in Section~\ref{subsect:eps_to_zero}, we have 
$$
\int_{R_0}^{R_b} n_k'(t,r) \, \phi'(r) + n_k''(t,r) \, \phi(r) \diff r = 0.
$$
We multiply by a smooth test function $\eta(t)$ and pass to the weak limit $\varepsilon \to 0$ to deduce
$$
\int_0^T \eta(t) \diff t \, \int_{R_0}^{R_b} (n_{\infty}'(r) \, \phi'(r) + n_{\infty}''(r) \, \phi(r)) \diff r = 0.
$$
As $n_{\infty} \in H^2(R_0, R_b)$ we deduce $n_{\infty}'(R_b) = 0$.
\end{proof}

\begin{lemma}\label{lem:f_UC_from_f2_UC}
Let $f:(a,b)\to \R$ be a continuous function such that $f^2$ is uniformly continuous. Then $|f|$ and $f$ are also uniformly continuous.
\end{lemma}

\begin{proof}
First, we observe that $|f|$ is uniformly continuous as a composition of a $\frac{1}{2}$-Hölder continuous function and a uniformly continuous one. Therefore,
\begin{equation}\label{eq:continuity_mod_f}
\forall {\varepsilon>0} \quad  \exists {\delta>0} \quad \forall {x,y \in (a,b)} \quad |x-y| \leq \delta \implies ||f(x)| - |f(y)|| \leq \varepsilon.
\end{equation}
Fix $\varepsilon>0$ and choose $\delta>0$ such that \eqref{eq:continuity_mod_f} holds with $\varepsilon/2$. Let $x, y \in (a,b)$ be such that $|x-y| \leq \delta$. If $f(x)$, $f(y)$ have the same sign we are done. Otherwise, by continuity, there exists $z$ between $x$ and $y$ such that $f(z) = 0$. As $|x-z|, |y-z| \leq \delta$, we can apply \eqref{eq:continuity_mod_f} again to deduce
$$
|f(x) - f(y)| \leq |f(x) - f(z)| + |f(z) - f(y)| \leq \varepsilon/2 + \varepsilon/2 = \varepsilon.
$$
\end{proof}

\section{Properties of the stationary states}
\label{sect:stationary_states}

The stationary solution built previously has compact support for $R_b$ large enough.  This is the main content of Theorem~\ref{thm:finalthm2} which we prove here. We still use, to simplify notations, the potential $V(r)=r^{2}$. We postpone to Appendix~\ref{app:generalforce} the case of a more general potential $V(r)$.

\subsection{Proof of Theorem~\ref{thm:finalthm2} (A)} 

We recall that, from Theorem \ref{eq:thm1_stationary}, $n_\infty \geq 0$ is $C^1$, $n_\infty'(R_{b})=0, n_\infty'(0)=0$. 

\begin{proof}[Proof of Theorem~\ref{thm:finalthm2} (A)]
To prove that $n_\infty$ is non-increasing, the main idea is to show that it cannot have a local maximum except at the point $r=0$. 

To do so, by contradiction, we assume there is  local maximum at $R_2 \in (0,R_{b}]$. This implies that $n_{\infty}'(R_{2})=0$, $n_{\infty}''(R_{2})\le 0$. Also by $C^1$ regularity, in a neighborhood of $R_2$ the equation hold 
\begin{equation*}
n_{\infty}^{\gamma}(r)-\f{\delta}{r}n_{\infty}'(r)-\delta n_{\infty}''(r)=C-r^{2},     
\end{equation*}
for some constant $C$. This equation implies that the local maximum is strict.

Also, still  by $C^1$ regularity, in this neighborhood of $R_2$ there is a point  $0< R_1 <R_2$ such that $0 <n_{\infty}(R_1)< n_{\infty}(R_2)$ and  $n_{\infty}'(R_{1})> 0$. Evaluating the equation at the points $R_{1}$ and $R_{2}$, and eliminating the constant $C$,  we obtain 
\begin{equation*}
\delta n_{\infty}''(R_{1})=R_{1}^{2}-R_{2}^{2}+n_{\infty}^{\gamma}(R_{1})-n_{\infty}^{\gamma}(R_{2})-\f{\delta}{R_{1}}n_{\infty}'(R_{1})+\delta n_{\infty}''(R_{2})< 0.     
\end{equation*}
Therefore $n_{\infty}$ is strictly concave at $R_{1}$. Consequently, $R_1$ can be continued to smaller values, $n_{\infty}(R_1)$ staying concave increasing (and thus $n_{\infty}'$ larger and larger as $R_1$ decreases)  until either $R_1=0$ or $n_{\infty}(R_1)=0$. In both cases we get  a contradiction with the condition $n_{\infty}'(R_1)=0$ which holds at $0$ and at values where $n_{\infty}(R_1)=0$. 
 
Consequently, the only possible local maximum is at $0$ and $n_{\infty}$ is non-increasing.
\\

The upper bound on $n_\infty(R_b)$ is just to say that $n_\infty(r) \geq n_\infty(R_b)$ on the full interval.

\end{proof}

\subsection{Proof of Theorem~\ref{thm:finalthm2} (B)}

We now consider a stationary state such that $n_{\infty}(R_{b})=0$. Theorem~\ref{thm:finalthm2} (A) asserts that there is $R\in[0,R_{b}]$ such that $n_{\infty}(r)=0$ on $[R,R_{b}]$ and $n_{\infty}$ is positive on $[0,R)$. Hence, on $[0,R]$, the relation~\eqref{eq:stationarystates1} shows that there exists a constant, that we write $R^{2}-\lambda$, such that $n_{\infty}$ solves  

\begin{equation}\label{eq:model_positive_part}
\begin{cases}
n^{\gamma}(r)-\f{\delta}{r}n'(r)-\delta n''(r)=R^{2}-r^{2}-\lambda , \qquad 0 \leq r \leq R, 
\\[8pt]
n(R) = 0.
\end{cases}
\end{equation}
Because it is $C^1$, the stationary solution also satisfies $n_{\infty}'(R)=0$ (and this is also true for $R=R_b$ as stated in Theorem~\ref{thm:finalthm1}.
We prove that there exists only one value $\lambda$ such that the solution of Equation~\eqref{eq:model_positive_part} also satisfies the condition $n'(R)=0$.
\\

Firstly, we exclude some  values of $\lambda$. Here, we use the notation $n^{\gamma}$ for $\max(0,n)^{\gamma}$.

\begin{lemma} \label{lm:extremelambda}
Being given $\lb \in \R$, let $n$ be the solution of Equation~\eqref{eq:model_positive_part}
Then, we have
\begin{itemize}
    \item when $\lambda \ge R^2$,  \qquad $n(r) \leq 0$ \quad $\forall r\in[0,R]$ \quad and \quad$n'(R) > 0$,
    \item when $\lambda \le \, 0$, \, \qquad $n(r) \geq 0$ \quad  $\forall r\in[0,R]$ \quad and \quad $n'(R) < 0$. 
\end{itemize}
\end{lemma}
\begin{proof}

For $\lambda \ge R^2$,  $n \leq 0$ is a consequence of the maximum principle and it follows immediately that $n'(R)\ge 0$. If we had  $n'(R)=0$,  the equation gives  $n''(R)=\f{\lambda}{\delta}>0$ which is in contradiction with the fact that $n$ is nonpositive in a small left neighborhood of $R$. 
\\

\noindent For $\lambda \le 0$, $n \geq 0$ is a consequence of the maximum principle and it follows that $n'(R) \leq 0$. To exclude the possibility that $n'(R) = 0$, we suppose by contradiction that $n'(R)=0$. Since we also have $n(R)=0$, we find that $n''(R)=\frac{\lambda}{\delta}$. As before, for $\lambda < 0$, we find contradiction. For $\lambda = 0$, we have $n''(R) =0$. Differentiating the equation, we find
$$
\gamma \, n^{\gamma-1}(r) \, n'(r) - \delta n^{(3)}(r) - \delta \frac{n''(r)}{r} + \delta \frac{n'(r)}{r^2}= -2r,
$$
and thus $n^{(3)}(R) = 2R/\delta > 0$. As $n(R)=n'(R) = n''(R)=0$, it follows that in a small neighbourhood of $R$, $n$ has to be negative raising a contradiction. The lemma is proved.
\end{proof}

Secondly, from Lemma~\ref{lm:extremelambda}, we may conclude that there is at least one value $\lambda \in (0,R^2)$ such that the Neumann condition is satisfied. This value is unique

\begin{lemma}\label{lem:lambda_exist_continuity}
There exists only one $\lambda \in (0,R^2)$ such that the solution of \eqref{eq:model_positive_part} satisfies $n'(R) = 0$. 
\end{lemma}

\begin{proof}
%
%
Suppose there are two solutions $n_{1}, n_{2}$ of \eqref{eq:model_positive_part} with $0<\lambda_{1}<\lambda_{2}<R^{2}$ such that $n_{i}(R)=n_{i}'(R)=0$ for $i=1,2$. From \eqref{eq:model_positive_part}, we find $n_{i}''(R)=\f{\lambda_{i}}{\delta}$. Therefore $0<n_{1}''(R)<n_{2}''(R)$ and we conclude by a Taylor expansion that $n'_{2}$ is smaller than $n'_{1}$ in a small left neighborhood of~$R$ which contradicts that $n$ decreases with $\lambda$. This proves Lemma~\ref{lem:lambda_exist_continuity}. 
\end{proof}
\begin{proof}[Proof of Theorem \ref{thm:finalthm2} (B)]
Clearly, $n_{\infty}$ is a solution to the problem \eqref{eq:model_positive_part} with some $\lambda$. By Lemma~\ref{lm:extremelambda}, we know that $\lambda \in (0,R^2)$ and then Lemma~\ref{lem:lambda_exist_continuity} yields the unique value of $\lambda$.
\\

For the second assertion, if there are two solutions $(n_1, \lambda_1)$, $(n_2, \lambda_2)$ of \eqref{eq:model}, Lemma~\ref{lem:lambda_exist_continuity} applies and we obtain that $\lambda_1 = \lambda_2$. The conclusion follows from uniqueness of solutions of the elliptic PDE~\eqref{eq:model_positive_part}.

\end{proof}

\subsection{Proof of Theorem \ref{thm:finalthm2} (C)}

Consider a solution $n_{\infty}$ of \eqref{eq:thm1_stationary} with $a:= n_{\infty}(R_{b}) > 0$. From Theorem~\ref{thm:finalthm2} (A), we know that $n_{\infty}$ is  $C^1$ and  $n_{\infty}'\leq 0$ so that $n_{\infty} > 0$. Therefore, from the equation for $n_\infty$, $rn_{\infty}(r)$ is $C^2$, and~\eqref{eq:thm1_stationary} boils down to
\begin{equation}\label{eq:PDE_stat_state_n>0}
\begin{cases}
n_{\infty}^{\gamma} - \f{\delta}{r}n'-\delta n'' = R_{b}^2 - r^2 - \lambda &\mbox{ in } (0,R_{b}),
\\
n_{\infty}(R_{b})=a >0, \, n'(R_{b}) = 0,\\
m = \int_0^{R_{b}} r\, n_{\infty}(r) \diff r, \qquad &
\end{cases}\end{equation}
where $\lambda$ is some constant. Our goal is to prove that if $R_{b}$ is sufficiently large with respect to $m$, there is no such solution of~\eqref{eq:PDE_stat_state_n>0}. Therefore we now assume that $R_{b}>2$.
\\

A useful formula in the sequel is, because of radial symmetry and after integration between~$0$ and~$r$,
\begin{equation}\label{eq:th2Cintegral}
\int_0^r \bar r\, n_{\infty}^{\gamma}(\bar r) \diff \bar r  - \delta r \p_r n_{\infty} = ({R_{b}}^2-\lb)\frac{r^2}{2}- \frac{r^4}{4}.
\end{equation}
Another useful general observation is that we may assume 
\[
n_{\infty}({R_{b}})^{\gamma} \leq {R_{b}}^2.
\]
Otherwise, by Theorem \ref{thm:finalthm2} (B), we have $m\geq \frac{{R_{b}}^{2+2/\gamma}}{2}$ which proves the result. 
\\

Firstly, we provide lower and upper bounds on admissible values of the constant $\lambda$
\begin{equation}\label{eq:lu_bound_lambda}
-{R_{b}}^2 \leq -n_{\infty}({R_{b}})^\gamma \leq \lb \leq \frac{{R_{b}}^2}{2}.
\end{equation}
The first inequality is the above restriction on $n_{\infty}({R_{b}})^\gamma$. The second inequality is valid because $n_{\infty}''({R_{b}}) \geq 0$ since $n_{\infty}$ is decreasing and $n_{\infty}'({R_{b}})=0$. The third inequality is just~\eqref{eq:th2Cintegral} at $r={R_{b}}$.\\

Secondly, we provide a control of $n_{\infty}(0)$. To do so, using~\eqref{eq:th2Cintegral}, $\p_r n_{\infty} \leq 0$ and the above upper bound on $\lb$,  we estimate~$|\p_r n_{\infty}|$ from above as 
\[
\delta |\p_r n_{\infty}| \leq {R_{b}}^2 r.
\]
This gives 
\[
n_{\infty}(r) \geq n_{\infty}(0) - \frac{{R_{b}}^2}{2\delta} r^2
\]
and, with $\al>0$ such that $\al^2=\frac{\delta }{2{R_{b}}^2} \leq 1 $,
\[
m \geq \int_0^{\al {R_{b}}} r n_{\infty}(r) \diff r \geq \frac{\al^2 {R_{b}}^2}{2} \left(n_{\infty}(0) - \frac{{R_{b}}^2}{2\delta}  \frac{\al^2 {R_{b}}^2}{2}\right) = \frac{\al^2 {R_{b}}^2}{2}\left(n_{\infty}(0) - \frac{{R_{b}}^2}{8} \right) .
\]
As a conclusion of this step, we may assume
\[
n_{\infty}(0) \leq \frac {{R_{b}}^2} {4} ,
\]
otherwise $m \geq \frac{\al^2 {R_{b}}^2}{2} \frac{{R_{b}}^2}{8} = \delta \frac{{R_{b}}^2}{32} $ and the result is proved again.
\\

Thirdly, we prove that with this control from above of $n_{\infty}(0)$, the derivative $|\p_r n_{\infty}|$ is large, thus again there is a control on the mass since $n_{\infty}$ is decreasing. To do so, we use again~\eqref{eq:th2Cintegral} and the third inequality in~\eqref{eq:lu_bound_lambda}. This gives 
\[
\delta r |\p_r n_{\infty}| \geq - n_{\infty}(0)^\gamma \frac{r^2}{2} + \frac{{R_{b}}^2}{2} \frac{r^2}{2}- \frac{r^4}{4}\geq \frac{r^2}{2} 
\left(- \frac {{R_{b}}^2} {4} + \frac{{R_{b}}^2}{2} -\frac{r^2}{2}  \right) = \frac{r^2}{4}  \left(\frac {{R_{b}}^2} {2} -r^2 \right)
\]
where we have used the smallness assumption on $n_{\infty}(0)$ and $\gamma \geq 1$. On the range $r\in (0, \frac {R_{b}} 2)$, we control 
\[
\delta |\p_r n_{\infty}| \geq \frac{r}{4} \frac {{R_{b}}^2} {4}, \quad \text{thus} \quad n_{\infty}(r)  \geq \frac {{R_{b}}^2} {32 \delta}\; \left(\frac{{R_{b}}^2}{4}- r^2\right),
\] 
and thus
\[
 m \geq \frac {{R_{b}}^2} {32 \delta}  \int_0^{{R_{b}}/2} \left(\frac{{R_{b}}^2}{4}- r^2\right) \diff r \geq \frac {{R_{b}}^5} {4 \cdot 128 \,\delta} .
\] 
Again we have the desired control and Theorem~\ref{thm:finalthm2} (C) is proved.

\section{Proof of Theorem \ref{thm:finalthm3}}
\label{sec:proof3}

To study the incompressible limit of stationary states of the Cahn-Hilliard equation, the difficulty comes from the singularity of the pressure.However it is possible to fully characterize them, and calculate the pressure jump at the tumor boundary. We begin with establishing the existence and uniqueness for the solution $n_{inc}$ of the limiting equation. Then, we show that all limits of $n_k$'s are determined by this profile $n_{inc}$.

\subsection{Preliminary steps}

If a sequence $\gamma_k \to \infty$ of stationary states $n_k$ converges to $n_{inc}$ and the sequence of pressures $p_k=n_k^{\gamma_k}$ converges to $p_{inc}$. Then we expect that $p_{inc}(n_{inc}-1)=0$. Therefore, there should be a 'tumor zone' where $n_{inc}=1$ and the pressure vanishes outside. 


This leads us to study the following problem in the zone $(R_0,R)$ where  $p_{inc}=0$:
\begin{equation}\label{eq:critical_problem}
\begin{cases}
- \f{\delta}{r}u_{c}'-\delta u_{c}'' = R^2 - r^2 - \lambda_c \qquad \mbox{in } (R_{0},R),
\\
u_{c}(R) = u_{c}'(R)= 0, \qquad \qquad u_c(R_{0}) = 1,  u_{c}'(R_{0}) = 0, 
\\
\int_{\Omega}r n_{inc}(r) \diff r= m,
\end{cases}
\end{equation}
where $n_{inc}$ is the extension of $u_{c}$ by $1$ on $[0,R_0]$.
\\

In a later subsection, we prove the convergence of the stationary states $n_k$ to this limiting profile. 
\\

Notice that System~\eqref{eq:critical_problem} has three free parameters ($R$, $R_0$, $\lambda_c$) and three constraints (2 additional boundary conditions and mass $m$). The following proposition gives the existence of a solution.
\begin{proposition}[Unique limiting profile]\label{mainthm3}
Let $m>72 \,\delta^{1/2}$. There exists uniquely determined $R>0$, $\lambda_c \in (0,R^2)$ and $R_0 \in (0,R)$ such that Equation~\eqref{eq:critical_problem} has a solution. 
 Furthermore,
 \begin{equation}\label{R0L_expansion}
     R_0 = \sqrt{R^2 - 2\lambda_c} \qquad \text{and} \qquad \lambda_c \approx \sqrt[3]{6} \, R^{2/3} \, \delta^{1/3} \qquad \text{for small }\delta > 0.
 \end{equation}
\end{proposition}

We postpone the proof of this proposition to the next subsection. Its proof uses an explicit solution obtained by the following problem. Find a couple $(\lambda_u, u)$ such that
\begin{equation}\label{eq:ref_problem}
\begin{cases}
- \f{\delta}{r}u'-\delta u'' = R^2 - r^2 - \lambda_u &\mbox{ in } (0,R)\\
u(R)=u'(R) = 0.
\end{cases}
\end{equation}
\begin{proposition}[Lower bound profile]\label{mainthm2}
Let $\lambda_u \in [0,R^2]$, then the solution $u$ of \eqref{eq:ref_problem} satisfies

\begin{itemize}
    \item[(A)] the explicit formula for $u$  
    $$
    u(r) = \frac{R^2}{4\delta} (R^2 - 2\lambda_u) \, \ln\left(\frac{r}{R}\right) + \frac{(r^2 - R^2)^2}{16\, \delta}  + \frac{R^2-2\lambda_u}{8 \delta} (R^2-r^2), 
    $$
    $$
    u'(r) = \frac{(R^2 - r^2)(R^2 - r^2 - 2\lambda_u)}{4\delta\, r},
    $$
    
\item[(B)] the function $u(r)$ is decreasing and positive for $r$ such that $0<R^2 - r^2 < 2\lambda_u$,
\item[(C)] for a solution $n_\infty$ of \eqref{eq:model} as in Theorem \ref{thm:finalthm2}, if $\lambda_\infty \geq \lambda_u$ then $n_\infty(r) \geq u(r)$ for $r \in (0,R]$. 
\end{itemize}
\end{proposition}

\begin{proof}[Proof of Proposition \ref{mainthm2}.]
To prove (A), we first compute $u'(r)$ and $u''(r)$:
$$
u'(r) = \frac{R^2}{4\delta} (R^2 - 2\lambda_u) \, \frac{1}{r} + \frac{(r^2-R^2)\,r}{4\delta} - \frac{R^2-2\lambda_u}{4 \delta} r = \frac{(R^2 - r^2)(R^2 - r^2 - 2\lambda_u)}{4\delta\, r},
$$
$$
u''(r) = \frac{-4r(R^2-r^2) + 4\lambda_u r}{4\delta \, r} 
- \frac{(R^2 - r^2)(R^2 - r^2 - 2\lambda_u)}{4\delta\, r^2} = -\frac{1}{\delta} \, (R^2 - r^2 - \lambda_u) - \frac{u'(r)}{r}.
$$
Therefore, we obtain the desired equation~\eqref{eq:ref_problem}.
\\

The statement (B) is an immediate consequence of the formula for $u'(r)$. 
\\

Finally, we prove (C). We introduce $
h(r) = n_\infty(r) - u(r)
$
and we have to prove that $h(r) \geq 0$. From the equations we get
$$
n^{\gamma} - \f{\delta}{r}h'-\delta h'' = \lambda_u - \lambda \mbox{ in } (0,R].
$$
So, thanks to our assumptions and  letting $g'(r) = r h'(r)$, we have
$$
h''(r) + \frac{h'(r)}{r} \geq 0 , \qquad \quad g''(r) \geq 0.
$$
 Integrating this from $r$ to $R$ and using the boundary conditions, we obtain 
$$
g'(r) \leq 0 \implies r \, h'(r) \leq 0 \implies h'(r) \leq 0.
$$
Integrating this once again and using boundary conditions, we discover $h(r) \geq 0$ as desired.

\end{proof}

\subsection{Proof of Proposition \ref{mainthm3}.}

\noindent The explicit solution built in Proposition~\ref{mainthm2} allows us to characterize the parameters $\lambda_c$ and $R_0$. Indeed, we are looking for $\lambda_c$ and $R_0$ such that
\begin{equation}\label{eq:formula_uc'R_0}
u_c'(R_0) = \frac{(R^2 - R_0^2)(R^2 - R_0^2 - 2\lambda_c)}{4\delta\, r} = 0,
\end{equation}
\begin{equation}\label{eq:formula_ucR_0}
u_c(R_0) = \frac{R^2}{4\delta} (R^2 - 2\lambda_c) \, \ln\left(\frac{R_0}{R}\right) + \frac{(R_0^2 - R^2)^2}{16\, \delta}  + \frac{R^2-2\lambda_c}{8 \delta} (R^2-R_0^2) = 1.
\end{equation}
\begin{lemma}[Solving for $R_0$ and $\lambda_c$]\label{lem:R0_lambda_c}
Let $R>0$. Then \eqref{eq:formula_uc'R_0}-\eqref{eq:formula_ucR_0} has a unique solution if and only if $16\delta < R^4$. Moreover, the solution is given by
\begin{equation}\label{eq:formula_R0_lam}
R_0 = \sqrt{R^2 - 2\lambda_c}, \qquad \qquad \lambda_c = \frac{R^2 x_c}{2},
\end{equation}
where $x_c \in (0,1)$ is the unique solution of
\begin{equation}\label{eq:nonlinear_for_xc_lemma}
 (1 - x_c) \, \ln\left(1 - x_c\right) + \frac{1}{2} x_c^2   +  (1-x_c)\, x_c = \frac{8\delta}{R^4}.
\end{equation}
\end{lemma}

\begin{proof} We split the reasoning into several steps.\\

\noindent \textit{Step 1: Equation for $R_0$.} Because $R_0 = R$ cannot fit~\eqref{eq:formula_ucR_0},  from \eqref{eq:formula_uc'R_0} we immediately deduce the formmula for $R_0$ in~\eqref{eq:formula_R0_lam}.

\noindent \textit{Step 2: Equation for $\lambda_c$.} We plug the formula for $R_0$  into \eqref{eq:formula_ucR_0} to deduce
$$
\frac{R^2}{4\delta} (R^2 - 2\lambda_c) \, \ln\left(\frac{\sqrt{R^2 - 2\lambda_c}}{R}\right) + \frac{4\lambda_c^2}{16\, \delta}  + \frac{R^2-2\lambda_c}{8 \delta} 2\lambda_c = 1.
$$
Using properties of logarithm and simple algebra, we have
$$
\frac{R^2}{8\delta} (R^2 - 2\lambda_c) \, \ln\left(1 - \frac{2\lambda_c}{R^2}\right) + \frac{\lambda_c^2}{4\, \delta}  + \frac{R^2-2\lambda_c}{4 \delta} \lambda_c = 1.
$$
Introducing the auxiliary variable $x_c = \frac{2\lambda_c}{R^2}$ and after  multiplication by $\frac{8\delta}{R^4}$, this equation is equivalent to Equation~\eqref{eq:nonlinear_for_xc_lemma}.

\noindent \textit{Step 3: Existence and uniqueness of $x_c$ and $\lambda_c$.} We prove that if $16 \, \delta < R^4$, equation \eqref{eq:nonlinear_for_xc_lemma} has a unique solution. To this end, we define
\begin{equation}\label{eq:def_f}
f(x):= (1 - x) \, \ln\left(1 - x\right) - \frac{1}{2} x^2 + x, \qquad \quad f(0)=0, \quad f(1)=\frac 1 2.
\end{equation}
Then, we compute
\begin{equation}\label{eq:deriv_f}
f'(x) =  -\ln(1-x) - x, \qquad f''(x) = \frac{1}{1-x} - 1.
\end{equation}
Since $f'(0) = 0$ and $f''(x) > 0$ for $x \in (0,1)$, it follows that $f'(x) > 0$ so that $f(x)$ is increasing. It follows that $f$ is one-to-one from $(0,1)$ into $(0, \frac{1}{2})$. Therefore, when $16 \delta < R^4$, there exists a unique $x_c \in (0,1)$ such that $f(x_c) = \frac{8 \delta}{R^4}$.
\end{proof}

\begin{lemma}[Estimates for $x_c$]
Let $x_c$ be a solution to \eqref{eq:nonlinear_for_xc_lemma} and $16\, \delta < R^4$. Then, we have
\begin{equation}\label{eq:x_c_taylor}
x_c \approx 2 \sqrt[3]{6} \, \delta^{1/3} R^{-4/3}, \qquad \qquad \lambda_c \approx \sqrt[3]{6} \, \delta^{1/3} R^{2/3} \qquad \qquad (\mbox{ as } \delta \to 0).
\end{equation}
More precisely, we have
\begin{equation}\label{eq:x_c_upper_bound}
x_c \leq 2 \sqrt[3]{6} \, \delta^{1/3} R^{-4/3}.
\end{equation}
Moreover, if $6^4\,8\, \delta < R^4$, we have
\begin{equation}\label{eq:x_c_lower_bound}
x_c \geq 2 \sqrt[3]{5} \, \delta^{1/3} R^{-4/3}.
\end{equation}
\end{lemma}
\begin{proof}
For $\delta$ small, Equation \eqref{eq:nonlinear_for_xc_lemma} shows that $x_c$ is small. More precisely, using \eqref{eq:def_f}, \eqref{eq:deriv_f}, we obtain $f(0) = f'(0) = f''(0) = 0$ and $f^{(3)}(0) = 1$ since $f^{(3)}(x) = \frac{1}{(1-x)^2}$. Hence, by the Taylor expansion, for small $x$, $f(x) \approx \frac{x^3}{6}$. Plugging this approximation into \eqref{eq:nonlinear_for_xc_lemma}, we obtain Estimate~\eqref{eq:x_c_taylor}.
\\

Next, we observe that 
$$
f^{(k)}(x) = \frac{(k-2)!}{(1-x)^{k-1}}, \qquad f^{(k)}(0) = (k-2)!.
$$
In particular, the Taylor expansion around $x=0$ gives
$$
f(x) = \sum_{k \geq 3} \frac{(k-2)!}{k!} x^k = \sum_{k \geq 3} \frac{1}{k\,(k-1)} x^k.
$$
Therefore, $f(x)$ is controlled by
\begin{equation}\label{eq:lower_bound_x_with_1-x}
\frac{x^3}{6} \leq f(x) \leq \frac{x^3}{6} \sum_{k\geq 0} x^k = \frac{x^3}{6\,(1-x)}.
\end{equation}
The control \eqref{eq:x_c_upper_bound} follows from the lower bound. 
\\

Finally, using this, we can find $\delta$ such that $2 \sqrt[3]{6} \, \delta^{1/3} R^{-4/3} \leq \frac{1}{6}$, namely $6^4 8\delta \leq R^4$. Then, we have $x_c \leq \frac{1}{6}$ so that $1-x_c \geq \frac{5}{6}$ and then the estimate \eqref{eq:lower_bound_x_with_1-x} gives us
$$
\frac{8\delta}{R^4} = f(x_c) \leq \frac{x_c^3}{6(1-x_c)} \leq \frac{x_c^3}{5}
$$
so that
$$
\frac{8\delta}{R^4} \leq \frac{x_c^3}{5} \iff \frac{40 \delta}{R^4} \leq x_c^3 \iff 2 \sqrt[3]{5} \, \delta^{1/3} R^{-4/3} \leq x_c.
$$
\end{proof}

\begin{lemma}\label{lem:mass_formula}
Let $u_c$ and $n_{inc}$ be a as in Equation \eqref{eq:critical_problem}, 
then the total mass of $n_{inc}$ satisfies 
$$
\mathcal{M}(n_{inc}) := \int_0^R r \, n_{inc}(r) \diff r = \frac{R^6 \, x_c^3(R)}{96 \, \delta}.
$$
Moreover, the map $R \mapsto \frac{R^6 \, x_c^3(R)}{96 \, \delta}$ is increasing if $R^4 x_c^3(R) > 32\, \delta$.
\end{lemma}
\begin{proof}
Because $n_{inc}$ is a $C^1$ function, integrating by parts, we find
$$
\mathcal{M}(n_{inc})= \int_0^R r\, n_{inc}(r) \diff r = -\frac{1}{2} \int_0^R r^2 n_{inc}'(r) \diff r = -\frac{1}{2} \int_{R_0}^R r^2 u_c'(r) \diff r .
$$
Inserting the formula for $u_c'(r)$ stated in Proposition~\ref{mainthm2}, we deduce that
$$
\mathcal{M}(n_{inc})= -\frac{1}{8 \delta} \int_{R_0}^R r\,(R^2 - r^2)(R^2 - r^2 - 2\lambda_c) \diff r. 
$$
With the notations $\lambda_c := R^2 x_c/2$ and $R_0 = R \sqrt{1-x_c}$, we obtain
$$
\mathcal{M}(n_{inc})= -\frac{1}{8 \delta} \int_{R \sqrt{1-x_c}}^R r\,(R^2 - r^2)(R^2 - r^2 - R^2 x_c) \diff r.
$$
We change variables $\tau = R^2 - r^2$ to get the desired formula
$$
\mathcal{M}(n_{inc})= -\frac{1}{16 \delta} \int_{0}^{R^2 x_c} \tau (\tau - R^2 x_c) \diff \tau = \frac{R^6 x_c^3}{96\, \delta}.
$$
For the second assertion, it is sufficient to prove that the map $R \mapsto R^6 x_c^3(R)$ is strictly increasing. Note that $x_c(R)$ is given implicitly via equation \eqref{eq:nonlinear_for_xc_lemma}. Differentiating it with respect to $R$, we discover that
$$
\frac{\diff x_c}{\diff R} (x_c + \log(1-x_c)) = \frac{32\, \delta}{R^5} \implies \frac{\diff x_c}{\diff R} = \frac{32\, \delta}{R^5\, (x_c + \log(1-x_c))}.
$$
Then, we study the derivative of $ R^6 \, x_c^3(R)$,
$$
\frac{\diff(R^6 x_c^3(R))}{\diff R} = 6 R^5 x_c^3 + 3 R^6 x_c^2  \frac{\diff x_c}{\diff R} =
6 R^5 x_c^3 + \frac{96\, \delta\, R\, x_c^2}{(x_c + \log(1-x_c))}.
$$
Using a simple Taylor estimate, we have $\frac{1}{x + \log (1-x)} \geq \frac{-2}{x^2}$ and we conclude that $R^6 x_c^3(R)$ is increasing since 
$$
\frac{\diff(R^6 x_c^3(R))}{\diff R} \geq 6 R^5 x_c^3 - 192\, \delta\, R = 6 R\, (R^4 x_c^3 - 32 \delta).
$$
\end{proof}

\begin{proof}[Proof of Proposition~\ref{mainthm3}] First, we notice that if $(R_0, R, \lambda)$ satisfy conditions of the Proposition~\ref{mainthm3}, then $R_0$, $\lambda_c$ are given by \eqref{eq:formula_R0_lam} (Lemma \ref{lem:R0_lambda_c}) and $\frac{R^6 x_c^3}{96 \delta} = m$ (Lemma \ref{lem:mass_formula}). Then, by the control $6^2 \, 2\,  \delta^{1/2} \leq m$ as well as an upper bound on $x_c$, cf. \eqref{eq:x_c_upper_bound}, we deduce
$$
6^2 \, 2\,  \delta^{1/2} \leq m = \frac{R^6\, x_c^3}{96\, \delta} \leq \frac{R^2}{2} \implies 6^4 4^2 \delta \leq R^4.
$$ 
This means that we can apply the lower bound \eqref{eq:x_c_lower_bound} to deduce 
$$
R^4 x_c^3 \geq 40 \, \delta^{1/3}.
$$
It follows that the necessary condition for existence of $(R_0, R, \lambda)$ is $6^4 4^2 \delta \leq R^4$ which implies $R^4 x_c^3 \geq 40 \, \delta^{1/3}$. Therefore, by Lemma \ref{lem:mass_formula}, the map $R \mapsto R^6 x_c^3(R)$ is invertible and we can find uniquely determined $R$ such that
$$
m = \frac{R^6 x_c^3(R)}{96 \delta}.
$$
With such a value $R$ (because $16 \,\delta < R^4$), we can find unique $R_0$ and $\lambda_c$ solving \eqref{eq:formula_uc'R_0}-\eqref{eq:formula_ucR_0} so that the formula for the mass is satisfied and the conclusion follows.
\end{proof}
\noindent 

\subsection{Proof of Theorem \ref{thm:finalthm3}}
The solutions of Theorem~\ref{thm:finalthm2} satisfy, with $\lambda_{k}\in (0, R^2)$, $R_k>0$,
\begin{equation*}
\begin{cases}
n^{\gamma_{k}}_{k}-\delta n_{k}''-\frac{\delta}{r}n_{k}'=R_{k}^{2}-r^{2}-\lambda_{k}\quad \text{in $(0,R_{k})$},\quad n_{k}=0 \quad \text{in $(R_{k},R_{b})$},
\\
n_{k}(R_{k})=n_{k}'(R_{k})=0, \qquad \qquad n_{k}'(0)=0.
\end{cases}
\end{equation*}
Thanks to the maximum principle, the sequence $\{n_{k}^{\gamma_{k}}\}_{k}$ is bounded in $L^{\infty}(I_{R_{b}})$. Moreover, multiplying this equation by $n_k''$ and integrating by parts, we obtain 

\begin{equation*}
\int_{0}^{R_{k}} \left(\delta (n_{k}'')^{2}+\gamma n_{k}^{\gamma_{k}-1}(n_{k}')^{2}+\delta\f{(n_{k}')^{2}}{2r^{2}}\right) \diff r =\int_{0}^{R_{k}} n_k''\, (R_{k}^{2}-r^{2}-\lambda_{k}) \diff r.    
\end{equation*}
Since $0\le R_{k},\lambda_{k}\le R_{b}$,  the right-hand side is bounded by $\f{\delta}{2}\int_{0}^{R_{k}}(n_{k}'')^{2}+C(\delta,R_{b})$. 
Thus, $\{n_{k}''\}_{k}$ is uniformly bounded in $L^{2}(0,R_{b})$. Therefore, up to a subsequence, as $k \to \infty$
\begin{equation*}
n^{\gamma_{k}}_{k}\rightharpoonup p_{inc}\geq 0 \quad\text{weakly$^*$ in $L^{\infty}(I_{R_{b}})$},
\qquad \quad 
    n_{k}\to n_{inc} \leq 1 \quad \text{in $C^{1}(\overline{I_{R_b}})$}.
\end{equation*}

\noindent We also have the algebraic relation $p_{inc}(n_{inc}-1)=0$. The inequality $p_{inc}(n_{inc}-1)\le 0$ is  straightforward using $p_{inc}\ge 0$ and $n_{inc}\le 1$. It remains to show that $p_{inc}(n_{inc}-1)\ge 0$. For $\nu>0$, there exists $\gamma_{0}$ such that for $\gamma_k\ge\gamma_{0}$
\begin{equation*}
    n^{\gamma_{k}+1}_{k}\ge n^{\gamma_{k}}_{k}-\nu
\end{equation*}
 because the function $x\mapsto x^{\gamma}(x-1)$ is nonpositive on $[0,1]$ and attains its minimum $-\left(\frac{\gamma}{\gamma+1}\right)^\gamma \frac{1}{\gamma+1} \to 0$ as $\gamma\to \infty$. Then, from the strong convergence of $n_{k}$ and the weak convergence of $n^{\gamma_{k}}_{k}$ we know that $n_{k}^{\gamma_{k}}n_{k}$ converges weakly  to $p_{inc}\,n_{inc}$. Passing to the limit, we obtain 
\begin{equation*}
p_{inc}\,n_{inc}\ge p_{inc}-\nu,
\end{equation*}
for every $\nu>0$. Letting $\nu\to 0$ yields the result.\\ 
 
\noindent Since $\{\lambda_{k}\}_{k}$ and $\{R_{k}\}_{k}$ are also bounded subsequences, we can extract converging subsequences to ${\lambda}$ and  ${R}$ respectively. Thanks to the $C^{1}$ convergence we know that $n$ satisfies the boundary condition $n_{inc}(R)=n_{inc}'(R)=0$ and $n_{inc}$ is radially decreasing as the uniform limit of radially decreasing functions. Finally, we can pass to the limit in the equation of mass conservation and obtain $\int_{0}^{{R}} rn_{inc} \diff r=m$. To sum up, in the limit we obtain a $C^1$, nonincreasing function $n_{inc}$ satisfying
\begin{equation*}
\begin{cases}
p_{inc}- \f{\delta}{r}n_{inc}'-\delta n_{inc}'' = {R}^2 - r^2 - {\lambda} &\mbox{ in } (0,R),\\
n_{inc}(R)=n_{inc}'(R)=n_{inc}'(0)= 0, \\
\int_{0}^{R} r\,n_{inc}(r) \diff r= m,\\
p_{inc}(n_{inc}-1)=0.
\end{cases}
\end{equation*}

 The limiting ODE is satisfied on $(0,{R})$ because the ODE for $n_k$ is satisfied on $(0, \inf_{l \geq k} R_k)$. Passing to the limit, we obtain the ODE on $(0, \lim_{k \to \infty} \inf_{l \geq k} R_k) = (0,{R})$ because ${R}=\lim_{k \to \infty} R_k$.\\

 We claim that $n_{inc}$ reaches the value $1$. By contradiction, if $n_{inc} < 1$ on $[0,R]$, then $p_{inc} = 0$ so that $n_{inc}$ is a $C^1$ solution to the following ODE on $[0,R]$:
$$
- \f{\delta}{r}n_{inc}'-\delta n_{inc}'' = R^2 - r^2 - \lambda, \qquad n_{inc}(R) = n_{inc}'(R) = n_{inc}'(0) = 0. 
$$
By Proposition \ref{mainthm2} (A) such a solution does not exists.
\\

 By monotonicity and the fact that $n_{inc}$ reaches value 1 we deduce that there are two zones. In the zone $\{p_{inc}>0\}$ we have $n_{inc}=1$, and thus $p_{inc}=R^{2}-r^{2}-\lambda$. Then, when $n_{inc}<1$ ($n_{inc}$ is decreasing), let us say at $r=R_{0}$ we have $p=0$. The pressure jump  is equal to $ \llbracket p_{inc}\rrbracket=R^{2}-R_{0}^{2}-\lambda.$\\

\noindent Finally, the convergence of the whole sequence follows from uniqueness of the limiting profile as stated in Proposition~\ref{mainthm3}.

\section{Conclusion and numerical simulations}\label{sec:num}

Motivated by the pressure jump imposed in free boundary problems of tissue growth, \cite{Friedman,Perthame-Hele-Shaw, KT18, LiuXu}, we included surface tension in such compressible models. We established that radially symmetric stationary solutions of the Cahn-Hilliard system with a confining potential $V(r)$ exist and are decreasing. In the incompressible limit, they present a jump of pressure at the boundary of the saturation set $\{n=1\}$. We computed explicitly this pressure jump which is proportional to $\delta^{1/3}V'(R)^{2/3}$. There is a vacuum zone $\{n=0\}$ that induces a degeneracy which is the main difficulty when establishing the a priori estimates.

It is an open question to prove a similar result, for a propagating wave, when the system is driven by  a source term rather than a confining potential, as in~\cite{elbar-perthame-poulain} for instance. However, we provide numerical simulations in radial coordinates. More precisely, we focus on the system

\begin{equation}\label{eq:source_term}
\begin{aligned}
&\f{\p (rn)}{\p t}-\f{\p}{\p r}\left(rn\f{\p\mu}{\p r}\right)=nG(p),\quad \text{in}\quad (0,+\infty)\times I_{R_{b}},\\
&\mu =p-\f{\delta}{r}\f{\p}{\p r}\left(r\f{\p n}{\p r}\right),\qquad p=n^{\gamma}.
\end{aligned}
\end{equation}

When $\gamma\to\infty$, we expect to find the incompressible limit

\begin{equation*}
\begin{cases}
-\f{\p}{\p r}\left(r\f{\p p}{\p r}\right) =G(p), \qquad p(n-1)=0,  \quad\text{in $\{n=1\}$}, \\
\llbracket p \rrbracket=-\delta\llbracket\f{1}{r}\f{\p}{\p r}\left(r\f{\p n}{\p r}\right)\rrbracket\quad\text{on $\partial\{n=1\}$}.
\end{cases}
\end{equation*}

These equations are obtained formally after setting $n=1$ in~\eqref{eq:source_term} and using the relation $p(n-1)=0$. The main open question is to link the value of the pressure jump to the other parameters of the model, \ie the source term $G$, the parameter $\delta$ and boundary's curvature. In radial settings the curvature is $\f{1}{R(t)}$ where $R(t)$ is the radius of the tumor. We present below some numerical simulations for the evolution of the density and the pressure of the tumor. If the pressure jump seems to be decreasing as the tumor grows, it is not numerically clear how to determine the pressure jump. \\

\begin{figure}[H]
    \centering
    \begin{subfigure}[b]{0.45\textwidth}
        \includegraphics[width=\textwidth]{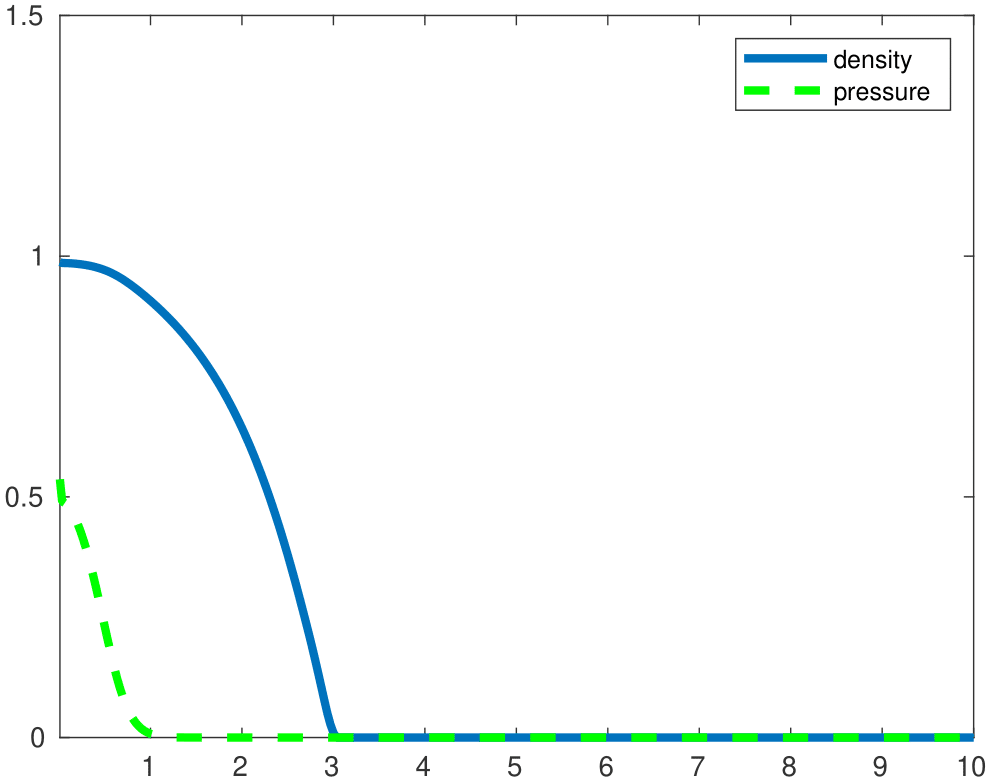}
        \caption{Initial condition}
    \end{subfigure}
    \begin{subfigure}[b]{0.45\textwidth}
        \includegraphics[width=\textwidth]{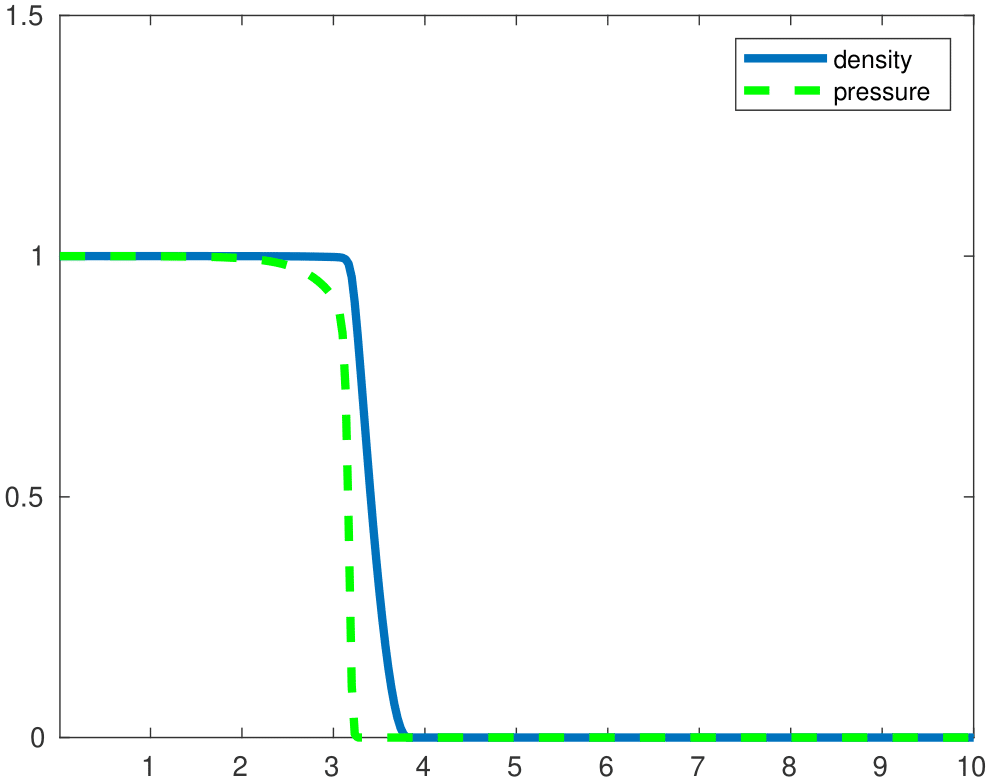}
        \caption{Evolution at $t=0.31$}
    \end{subfigure}
        \begin{subfigure}[b]{0.45\textwidth}
        \includegraphics[width=\textwidth]{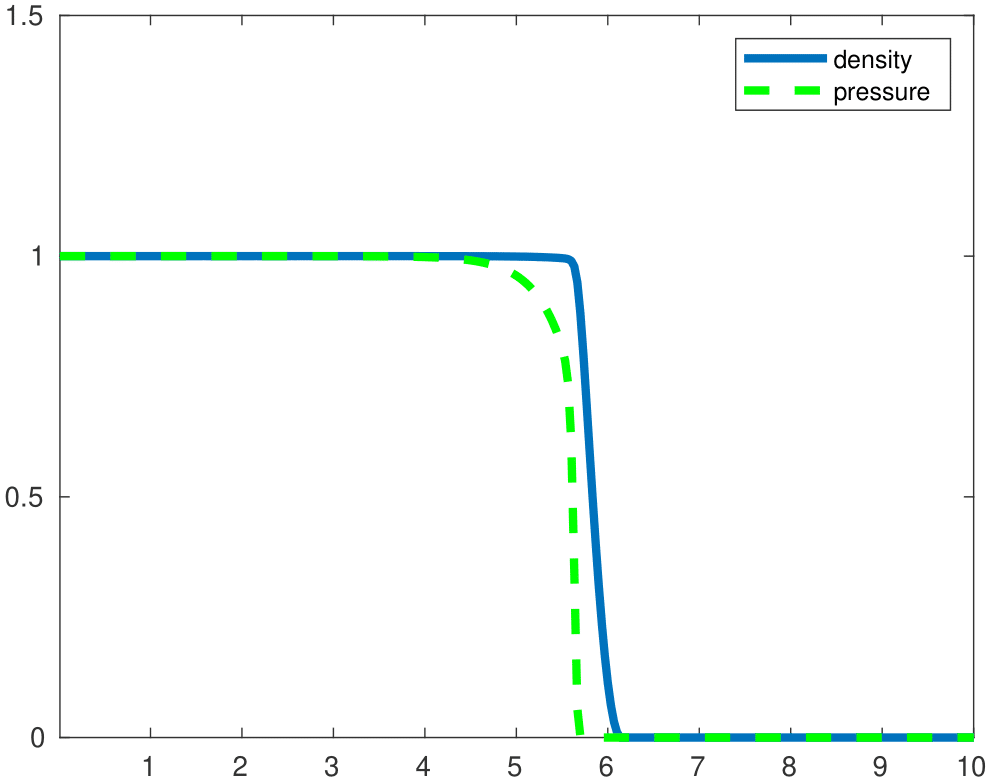}
        \caption{Evolution at $t=1.14$}
    \end{subfigure}
    \begin{subfigure}[b]{0.45\textwidth}
        \includegraphics[width=\textwidth]{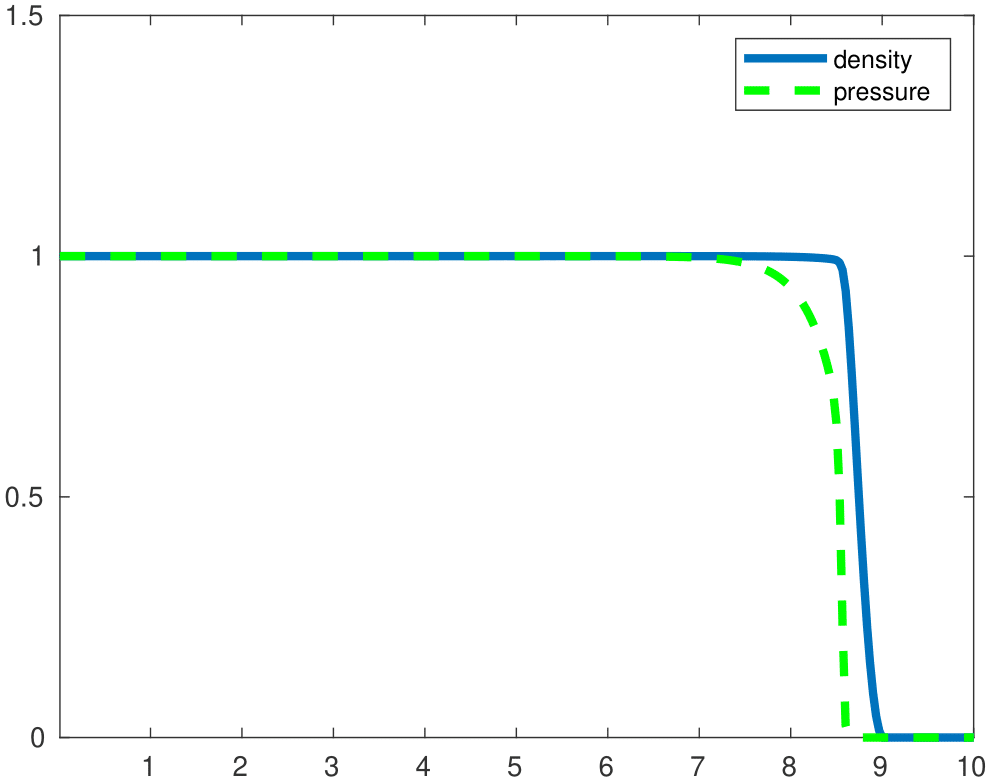}
        \caption{Evolution at $t=2.11$}
    \end{subfigure}
\end{figure}

{\bf Numerical settings.} For the source term, we take $G(p)=10(1-p)$. We use an explicit scheme, with time step $dt=1$e$-7$, final time $t=2.11$ and the interval is $I_{R_{b}}=[0,10]$ with 300 points. The initial condition is a truncated arctangent. To remove the degeneracy $r=0$ in the numerical scheme, we consider $r+\eps$ instead of $r$ for some small $\eps>0$.

\noindent The pressure $p$ reaches the value 1, as the density, because we choose the homeostatic pressure $p_{h}=1$ in the source term $G(p)=10(p_{h}-p)$. The homeostatic pressure is interpreted as the lowest level of pressure that prevents cell multiplication due to contact-inhibition. 

\section*{Acknowledgements}
B.P. has received funding from the European Research Council (ERC) under the European Union's Horizon 2020 research and innovation program (grant agreement No 740623). J.S. was supported by the National Science Center grant 2017/26/M/ST1/00783.

\appendix

\section{Limit profile for general force in dimension 2}
\label{app:generalforce}
\noindent We generalize the pressure jump formula obtained when $V(r) = r^2$ and, for a general strictly increasing force $V$, we establish that
\begin{equation}\label{lambda_gen}
\lambda \approx \frac{\sqrt[3]{12}}{2} \,\delta^{1/3} \, (V'(R))^{2/3}.
\end{equation}
Hence, we consider the solution in $(0,R)$ of
\begin{equation}\label{eq:PDE_elliptic_general_V}
\delta n''(r) + \delta \frac{n'(r)}{r} = = V(R) - V(r) - \lambda, \qquad n'(R) = n(R)=0.
\end{equation}
\\
\noindent \textit{Step 1: An expression of the solution.} The solution is given by
\begin{equation}\label{eq:eq_for_sol_gen_force} 
\begin{cases}
n(r) = \frac{R^2}{2\delta} \, (V(R) - \lambda) \, \log\left(\frac{r}{R}\right)  + \frac{R^2-r^2}{4\delta} (V(R)-\lambda) + \frac{1}{\delta} \, \int_r^R \frac{\mathcal{H}(z)}{z} \diff z,
\\
\mathcal{H}(z) := \int_z^R u\, V(u) \diff u. 
\end{cases}
\end{equation}
Indeed, we immediately verify that  $n(R) = 0$. Moreover, we have
\begin{equation}\label{eq:eq_for_der_gen_force}
n'(r) = \frac{R^2}{2\delta} \, (V(R) - \lambda) \, \frac{1}{r} - \frac{r}{2\delta} (V(R)-\lambda) - \frac{\mathcal{H}(r)}{\delta\,r} .  
\end{equation}
As $\mathcal{H}(R) = 0$, we have $n'(R) = 0$. Finally, we compute $n''$ using \eqref{eq:eq_for_der_gen_force}
\begin{equation}\label{eq:formula_for_n_second_app}
n''(r) = -\frac{R^2}{2\delta} \, (V(R) - \lambda) \, \frac{1}{r^2} -  \frac{1}{2\delta} (V(R)-\lambda) + \frac{1}{\delta} V(r) + \frac{\mathcal{H}(r)}{\delta\,r^2}.
\end{equation}
Therefore, combining \eqref{eq:eq_for_der_gen_force} and \eqref{eq:formula_for_n_second_app}, we obtain  \eqref{eq:PDE_elliptic_general_V}.
\\

\noindent \textit{Step 2: Limit profile.} We are looking for the solution $n$ of  \eqref{eq:PDE_elliptic_general_V} such that $n(R_0) = 1$, $n'(R_0) = 0$ for some $R_0 < R$ and some $\lambda$, i.e. we have two parameters $R_0$ and $\lambda$ to be found.\\

The {\em condition $n'(R_0) = 0$} is immediately obtained from \eqref{eq:eq_for_der_gen_force}. It is given by 
\begin{equation}\label{eq:gen_force_for_lambda}
(R^2 - R_0^2) \, (V(R) - \lambda) - 2\,\mathcal{H}(R_0) = 0.
\end{equation}
For the {\em condition $n(R_0) = 1$}, from \eqref{eq:eq_for_sol_gen_force} we obtain 
$$
\frac{R^2}{2\delta} \, (V(R) - \lambda) \, \log\left(\frac{R_0}{R}\right)  + \frac{R^2-R_0^2}{4\delta} (V(R)-\lambda) + \frac{1}{\delta} \, \int_{R_0}^R \frac{\mathcal{H}(z)}{z} \diff z = 1.
$$
Then, we use \eqref{eq:gen_force_for_lambda} to remove the term $(V(R)-\lambda)$, and multiply by $2 \delta$ to get an equation for $R_0$
$$
{R^2} \, \frac{ \mathcal{H}(R_0)}{R^2 - R_0^2} \, \log\left(\frac{R_0^2}{R^2}\right)  + \mathcal{H}(R_0) + 2 \, \int_{R_0}^R \frac{\mathcal{H}(z)}{z} \diff z = 2\delta.
$$

 We introduce the variable $\tau := \frac{R^2 - R_0^2}{R^2}$ so that the equation reads
\begin{equation}\label{eq:exact_equation_for_tau}
\mathcal{H}(\sqrt{1-\tau} R) \left( \frac{\log(1-\tau)}{\tau}\, + 1\right) + 2 \, \int_{\sqrt{1-\tau} R}^R \frac{\mathcal{H}(z)}{z} \diff z = 2\delta, \qquad \tau := \frac{R^2 - R_0^2}{R^2}.
\end{equation}

\noindent \textit{Step 3: Existence and uniqueness of $\tau$ and $R_0$.} We define the function
$$
\mathcal{F}(\tau):= \mathcal{H}(\sqrt{1-\tau} R) \left( \frac{\log(1-\tau)}{\tau}\, + 1\right) + 2 \, \int_{\sqrt{1-\tau} R}^R \frac{\mathcal{H}(z)}{z} \diff z.
$$
As $\mathcal{H}(R) = 0$ and $\frac{\log(1-\tau)}{\tau}$ is bounded near $\tau = 0$, we have $\mathcal{F}(\tau) = 0$. Now, we want to compute $\mathcal{F}'(\tau)$. First,
\begin{equation}\label{eq:der_H_sqrt}
\frac{\diff}{\diff \tau} \mathcal{H}(\sqrt{1-\tau} R) = \sqrt{1-\tau} R \, V(\sqrt{1-\tau} R) \, \frac{R}{2\sqrt{1-\tau}} = \frac{R^2\,V(\sqrt{1-\tau} R)}{2},
\end{equation}
$$
\frac{\diff}{\diff \tau} \left(2 \, \int_{\sqrt{1-\tau} R}^R \frac{\mathcal{H}(z)}{z} \diff z\right) = 2 \left(- \frac{\mathcal{H}(\sqrt{1-\tau} R)}{\sqrt{1-\tau} R} \right)
\, \frac{R}{2\sqrt{1-\tau}} = \frac{\mathcal{H}(\sqrt{1-\tau} R)}{1-\tau},
$$
$$
\frac{\diff}{\diff \tau}\left( \frac{\log(1-\tau)}{\tau}\, + 1\right) = \frac{1}{\tau\, (\tau - 1)} - \frac{\log(1-\tau)}{\tau^2}.
$$
Therefore, by the product rule, we find
\begin{equation}\label{eq_der_F_general_force}
\begin{split}
&\mathcal{F}'(\tau) =\\
&=\frac{R^2\,V(\sqrt{1-\tau} R)}{2} \left( \frac{\log(1-\tau)}{\tau}\, + 1\right) + \mathcal{H}(\sqrt{1-\tau} R)\left(\frac{1}{1-\tau} + \frac{1}{\tau\, (\tau - 1)} - \frac{\log(1-\tau)}{\tau^2}\right)\\
&= \frac{1}{\tau} \left( \frac{\log(1-\tau)}{\tau}\, + 1\right) \left(\frac{R^2\,\tau\,V(\sqrt{1-\tau} R)}{2} - {\mathcal{H}(\sqrt{1-\tau} R)} \right)
\end{split}
\end{equation}
Since $\frac{\log(1-\tau)}{\tau}\, + 1 < 0$ for $\tau \in (0,1)$, to prove $\mathcal{F}'(\tau)>0$, it is sufficient that for $\tau \in (0,1)$
\begin{equation}\label{eq:aux_fcn_R-H}
\frac{R^2\,\tau\,V(\sqrt{1-\tau} R)}{2} - {\mathcal{H}(\sqrt{1-\tau} R)} < 0.
\end{equation}
This function vanishes in $\tau = 0$. Moreover, its derivative with respect to $\tau$ is equal to
\begin{equation}\label{eq:aux_fcn_R-H_der}
\frac{R^2\,V(\sqrt{1-\tau} R)}{2} + \frac{R^2\,\tau\,V'(\sqrt{1-\tau} R)}{2} \, \frac{-R}{2\, \sqrt{1-\tau}} - \frac{R^2\,V(\sqrt{1-\tau} R)}{2} = -\frac{R^3 \tau V'(\sqrt{1-\tau} R)}{4\sqrt{1-\tau}}
\end{equation}
where we used \eqref{eq:der_H_sqrt}. As $V' > 0$, we conclude \eqref{eq:aux_fcn_R-H} which implies $\mathcal{F'}(\tau) > 0$. Hence, in some neighborhood of $0$ we can find exactly one $\tau$ that solves the equation. Moreover, it is unique as $\mathcal{F}$ is strictly increasing. Then the uniqueness of $R_0$ and $\lambda$ follows.\\

\noindent \textit{Step 4: Taylor expansion of $\tau$.} As $\mathcal{F}$ is strictly increasing, we expect the solution $\tau$ to be small (if $\delta$ is small). This justifies to use Taylor expansion around $\tau = 0$. We already know that $\mathcal{F}(0) = \mathcal{F}'(0) = 0$. Now, we claim $\mathcal{F}''(0) = 0$. Indeed,
$$
\frac{\diff}{\diff \tau}  \left( \frac{\log(1-\tau)}{\tau^2}\, + \frac{1}{\tau}\right) = -\frac{2\log(1-\tau)}{\tau^3} + \frac{\tau-2}{(1-\tau)\,\tau^2} = -\frac{2\log(1-\tau) (1-\tau) + (2-\tau)\, \tau}{(1-\tau)\, \tau^3}
$$
so that using \eqref{eq_der_F_general_force} and \eqref{eq:aux_fcn_R-H_der} we compute
\begin{equation}\label{eq:2nd_der_F_gen_force}
\begin{split}
\mathcal{F}''(\tau) =& -\frac{1}{\tau} \left( \frac{\log(1-\tau)}{\tau}\, + 1\right) \frac{R^3 \tau V'(\sqrt{1-\tau} R)}{4\sqrt{1-\tau}}\\
&-\frac{2\log(1-\tau) (1-\tau) + (2-\tau)\, \tau}{(1-\tau)\, \tau^3} \, \left(\frac{R^2\,\tau\,V(\sqrt{1-\tau} R)}{2} - {\mathcal{H}(\sqrt{1-\tau} R)} \right)
\end{split}
\end{equation}
Since $2\log(1-\tau) (1-\tau) + (2-\tau)\, \tau \approx \tau^3$ for small $\tau$, the expressions above are bounded in the neighborhood of $0$. Evaluating them at $\tau = 0$, we obtain $\mathcal{F}''(0) = 0$.\\

\noindent Now, we claim that $\mathcal{F}^{(3)}(0) \neq 0$. To see this, we write expression \eqref{eq:2nd_der_F_gen_force} in the form
$$
\mathcal{F}''(\tau) = A(\tau) \, B(\tau) + C(\tau) \, D(\tau)
$$
so that
$$
\mathcal{F}^{(3)}(\tau) = A'(\tau) \, B(\tau) + A(\tau) \, B'(\tau) + C'(\tau) \, D(\tau) + C(\tau) \, D'(\tau).
$$
We study the four terms above separately.
\begin{itemize}
    \item $A'(\tau)\, B(\tau)|_{\tau = 0} = 0$. Indeed, $A'(\tau) = -C(\tau)$ and the latter is bounded in the neighbourhood of 0 (see Taylor's expansion above). Moreover, $B(0) = 0$.
    \item $C(\tau) \, D'(\tau)|_{\tau = 0} = 0$. Indeed, $C(\tau)$ is bounded around $\tau = 0$, while $D'(0) = B(0) = 0$.
    \item $C'(\tau)\, D(\tau)|_{\tau = 0} = 0$. In fact, $D(0) = 0$, so it is sufficient to prove that $C'(\tau)$ is bounded near $\tau = 0$. We have
    $$
    C'(\tau) = \frac{\tau\,(2\tau^2 - 9\tau + 6) + 6\log(1-\tau) (1-\tau)^2}{\tau^4\,(1-\tau^2)}.
    $$
    Using expansion $\log(1-\tau) \approx - \tau - \frac{1}{2} \tau^2 - \frac{1}{3}\tau^3$ we have around $\tau = 0$:
    \begin{align*}
    \tau\,(2\tau^2 - &9\tau + 6) + 6\log(1-\tau) (1-\tau)^2 \approx\\
    &\approx
    2\tau^3 - 9\tau^2 + 6\tau - (6 \tau + 3 \tau^2 + 2\tau^3)(1 + \tau^2 - 2\tau)\\
    &= 2\tau^3 - 9\tau^2 + 6\tau - 6 \tau - 3 \tau^2 - 2\tau^3 - 6 \tau^3 - 3 \tau^4 - 2\tau^5 + 12\tau^2 + 6\tau^3 + 4\tau^4)\\
    &= \tau^4 - 2\tau^5.
    \end{align*}

\item $A(\tau) B'(\tau)|_{\tau=0} = \frac{R^3 V'(R)}{8}$. Indeed, $A(0) = \frac{1}{2}$. Moreover, we have
$$
B'(\tau) = \frac{R^3 \, V'(\sqrt{1-\tau} R)}{4\sqrt{1-\tau}} +
\frac{R^3 \tau V''(\sqrt{1-\tau} R)}{4\sqrt{1-\tau}} \frac{(-R)}{2\sqrt{1-\tau}} + \frac{R^3 \tau V''(\sqrt{1-\tau} R)}{8 (1-\tau)^{3/2}} 
$$
which implies $B'(0) = \frac{R^3 V'(R)}{8}$.
\end{itemize}

Therefore, equation \eqref{eq:exact_equation_for_tau} can be approximated with
$$
\frac{1}{6}\,\frac{R^3 V'(R)}{8}\, \tau^3 = 2 \delta \iff \tau^3 = \frac{96\, \delta}{R^3 V'(R)} \iff \tau = \frac{2 \sqrt[3]{12} \,\delta^{1/3}}{R \sqrt[3]{V'(R)}}.
$$
\noindent \textit{Step R: Taylor expansion for $\lambda$.} Small value of $\tau$ means that $R_0$ is close to $R$. Therefore, we expand in Taylor series equation \eqref{eq:gen_force_for_lambda} around $R_0 = R$. Writing
$$
\mathcal{G}(R_0):= (R^2 - R_0^2) \, (V(R) - \lambda) - 2\,\mathcal{H}(R_0) = 0,
$$
we have 
$$
\mathcal{G}'(R_0) = -2 R_0(V(R) -\lambda) + 2\, R_0 \, V(R_0) = 2 R_0\,(V(R_0) - V(R) + \lambda),
$$
$$
\mathcal{G}''(R_0) = 2 \,(V(R_0) - V(R) + \lambda) +  2R_0\,V'(R_0).
$$
It follows that $\mathcal{G}(R) = 0$, $\mathcal{G}'(R) = 2R_0 \lambda$ and $\mathcal{G}''(R) = 2\lambda + 2R V'(R)$. Therefore, \eqref{eq:gen_force_for_lambda} can be approximated with
$$
2R_0 \lambda (R_0-R) + \frac{1}{2} (2\lambda + 2R V'(R)) (R_0-R)^2 = 0,
$$
which can be rewritten as
$$
\lambda = \frac{R\,(R-R_0)\, V'(R)}{R_0 + R}.
$$
Using $\tau = \frac{R^2-R_{0}^2}{R^2}$ we have $R-R_0 = \frac{\tau R^2}{R+R_0}$ so that 
$$
\lambda = \frac{R^3\,\tau \, V'(R)}{(R_0 + R)^2} \approx \frac{R \, \tau \, V'(R)}{4} = \frac{\sqrt[3]{12}}{2} \,\delta^{1/3} \, (V'(R))^{2/3}.
$$

Notice that when $V(r) = r^2$, we have $V'(R) = 2R$. Therefore, $\lambda \approx \sqrt[3]{6} \, \delta^{1/3} R^{2/3}$ as in \eqref{R0L_expansion}.

\bibliographystyle{siam}
\bibliography{biblio}

\begin{thebibliography}{10}

\bibitem{Agosti-CH-2017}
{\sc A.~Agosti, P.~F. Antonietti, P.~Ciarletta, M.~Grasselli, and M.~Verani},
  {\em A {C}ahn-{H}illiard-type equation with application to tumor growth
  dynamics}, Math. Methods Appl. Sci., 40 (2017), pp.~7598--7626.

\bibitem{Alikakos_convergence}
{\sc N.~D. Alikakos, P.~W. Bates, and X.~Chen}, {\em Convergence of the
  {C}ahn-{H}illiard equation to the {H}ele-{S}haw model}, Arch. Rational Mech.
  Anal., 128 (1994), pp.~165--205.

\bibitem{Basan}
{\sc M.~Basan, T.~Risler, J.~Joanny, X.~Sastre‐Garau, and J.~Prost}, {\em
  Homeostatic competition drives tumor growth and metastasis nucleation}, HFSP
  Journal, 3 (2009), pp.~265--272.
\newblock PMID: 20119483.

\bibitem{Bittig_2008}
{\sc T.~Bittig, O.~Wartlick, A.~Kicheva, M.~Gonz{\'{a}}lez-Gait{\'{a}}n, and
  F.~Jülicher}, {\em Dynamics of anisotropic tissue growth}, New J. Phys., 10
  (2008), p.~063001.

\bibitem{BD}
{\sc H.~Byrne and D.~Drasdo}, {\em Individual-based and continuum models of
  growing cell populations: a comparison}, J. Math. Biol., 58 (2009),
  pp.~657--687.

\bibitem{byrne_modelling_2004}
{\sc H.~{Byrne} and L.~{Preziosi}}, {\em Modelling solid tumour growth using
  the theory of mixtures}, Math. Med. Biol., 20 (2003), pp.~341--366.

\bibitem{Cahn-Hilliard-1958}
{\sc J.~W. Cahn and J.~E. Hilliard}, {\em Free energy of a nonuniform system.
  i. interfacial free energy}, The Journal of Chemical Physics, 28 (1958),
  pp.~258--267.

\bibitem{chatelain_2011}
{\sc C.~Chatelain, T.~Balois, P.~Ciarletta, and M.~Ben~Amar}, {\em Emergence of
  microstructural patterns in skin cancer: a phase separation analysis in a
  binary mixture}, New J. Phys., 13 (2011), pp.~339--357.

\bibitem{finite_speed}
{\sc B.~Chen and C.~Liu}, {\em Finite speed of propagation for the
  cahn–hilliard equation with degenerate mobility}, Applicable Analysis, 100
  (2021), pp.~1693--1726.

\bibitem{Chen1996}
{\sc X.~Chen}, {\em {Global asymptotic limit of solutions of the Cahn-Hilliard
  equation}}, Journal of Differential Geometry, 44 (1996), pp.~262 -- 311.

\bibitem{ebenbeck_cahn-hilliard-brinkman_2018}
{\sc M.~Ebenbeck and H.~Garcke}, {\em On a {C}ahn-{H}illiard-{B}rinkman model
  for tumor growth and its singular limits}, SIAM J. Math. Anal., 51 (2019),
  pp.~1868--1912.

\bibitem{Ebenbeck-Brinkman}
{\sc M.~Ebenbeck, H.~Garcke, and R.~N\"{u}rnberg}, {\em
  Cahn-{H}illiard-{B}rinkman systems for tumour growth}, Discrete Contin. Dyn.
  Syst. Ser. S, 14 (2021), pp.~3989--4033.

\bibitem{elbar-perthame-poulain}
{\sc C.~Elbar, B.~Perthame, and A.~Poulain}, {\em Degenerate {C}ahn-{H}illiard
  and incompressible limit of a {K}eller-{S}egel model}, Accepted in Comm.
  Math. Sci.,  (2021).

\bibitem{Garcke-CH-deg}
{\sc C.~M. Elliott and H.~Garcke}, {\em On the {C}ahn-{H}illiard equation with
  degenerate mobility}, SIAM J. Math. Anal., 27 (1996), pp.~404--423.

\bibitem{Escher-solutions}
{\sc J.~Escher and G.~Simonett}, {\em Classical solutions for {H}ele-{S}haw
  models with surface tension}, Adv. Differential Equations, 2 (1997),
  pp.~619--642.

\bibitem{Frieboes-2010-CH}
{\sc H.~B. Frieboes, F.~Jin, Y.-L. Chuang, S.~M. Wise, J.~S. Lowengrub, and
  V.~Cristini}, {\em Three-dimensional multispecies nonlinear tumor
  growth---{II}: {T}umor invasion and angiogenesis}, J. Theor. Biol., 264
  (2010), pp.~1254--1278.

\bibitem{Friedman}
{\sc A.~Friedman}, {\em Mathematical analysis and challenges arising from
  models of tumor growth}, Math. Models Methods Appl. Sci., 17 (2007),
  pp.~1751--1772.

\bibitem{Frigeri-CH-Darcy}
{\sc S.~Frigeri, K.~F. Lam, E.~Rocca, and G.~Schimperna}, {\em On a
  multi-species {C}ahn-{H}illiard-{D}arcy tumor growth model with singular
  potentials}, Commun. Math. Sci., 16 (2018), pp.~821--856.

\bibitem{Garcke-2018-multi-CH-darcy}
{\sc H.~Garcke, K.~F. Lam, R.~N\"{u}rnberg, and E.~Sitka}, {\em A multiphase
  {C}ahn-{H}illiard-{D}arcy model for tumour growth with necrosis}, Math.
  Models Methods Appl. Sci., 28 (2018), pp.~525--577.

\bibitem{Garcke-2016-CH-darcy}
{\sc H.~Garcke, K.~F. Lam, E.~Sitka, and V.~Styles}, {\em A
  {C}ahn-{H}illiard-{D}arcy model for tumour growth with chemotaxis and active
  transport}, Math. Models Methods Appl. Sci., 26 (2016), pp.~1095--1148.

\bibitem{JPDE-7-77}
{\sc Y.~Jingxue}, {\em On the {C}ahn-{H}illiard equation with nonlinear
  principal part}, Journal of Partial Differential Equations, 7 (1994),
  pp.~77--96.

\bibitem{KT18}
{\sc I.~Kim and O.~Turanova}, {\em Uniform convergence for the incompressible
  limit of a tumor growth model}, Ann. Inst. H. Poincar\'e Anal. Non
  Lin\'eaire, 35 (2018), pp.~1321--1354.

\bibitem{LiuXu}
{\sc J.-G. Liu and X.~Xu}, {\em Existence and incompressible limit of a tissue
  growth model with autophagy}, SIAM J. Math. Anal., 53 (2021), pp.~5215--5242.

\bibitem{Lowengrub-bridging}
{\sc J.~Lowengrub, H.~Frieboes, F.~Jin, Y.~Chuang, X.~Li, P.~Macklin, S.~Wise,
  and V.~Cristini}, {\em Nonlinear modelling of cancer: bridging the gap
  between cells and tumours.}, Nonlinearity, 23 1 (2010), pp.~R1--R9.

\bibitem{MR4001523}
{\sc A.~Miranville}, {\em The {C}ahn-{H}illiard equation. Recent advances and
  applications}, vol.~95 of CBMS-NSF Regional Conference Series in Applied
  Mathematics, Society for Industrial and Applied Mathematics (SIAM),
  Philadelphia, PA, 2019.
\newblock Recent advances and applications.

\bibitem{perthame_poulain}
{\sc B.~Perthame and A.~Poulain}, {\em Relaxation of the {C}ahn-{H}illiard
  equation with singular single-well potential and degenerate mobility},
  European J. Appl. Math., 32 (2021), pp.~89--112.

\bibitem{Perthame-Hele-Shaw}
{\sc B.~Perthame, F.~Quir\'{o}s, and J.~L. V\'{a}zquez}, {\em The {H}ele-{S}haw
  asymptotics for mechanical models of tumor growth}, Arch. Ration. Mech.
  Anal., 212 (2014), pp.~93--127.

\bibitem{Perthame-incompressible-visco}
{\sc B.~Perthame and N.~Vauchelet}, {\em Incompressible limit of a mechanical
  model of tumour growth with viscosity}, Philos. Trans. Roy. Soc. A, 373
  (2015), pp.~20140283, 16.

\bibitem{sciume}
{\sc G.~Scium{\'e}}, {\em Mechanistic modeling of vascular tumor growth: an
  extension of biot's theory to hierarchical bi-compartment porous medium
  systems}, Acta Mech., 232 (2021), pp.~1445--1478.

\bibitem{Yin2001}
{\sc J.~Yin and C.~Liu}, {\em Radial symmetric solutions of the
  {C}ahn-{H}illiard equation with degenerate mobility.}, Electronic Journal of
  Qualitative Theory of Differential Equations [electronic only], 2001 (2001),
  pp.~Paper No. 2, 14 p., electronic only--Paper No. 2, 14 p., electronic only.

\end{thebibliography}

\end{document}